\documentclass{article}
\usepackage{MnSymbol}
\usepackage{mathrsfs} 
\usepackage{amsmath,amsthm}
\usepackage[all]{xy}
\usepackage{hhline}
\usepackage{supertabular}

\newtheorem{theorem}{Theorem}[section]
\newtheorem{corollary}[theorem]{Corollary}
\newtheorem{lemma}[theorem]{Lemma}
\newtheorem{proposition}[theorem]{Proposition}
\theoremstyle{definition}
\newtheorem{remark}[theorem]{Remark}
\newtheorem{definition}[theorem]{Definition}

\numberwithin{equation}{section}

\DeclareMathAlphabet\mathbb{U}{msb}{m}{n}

\newcommand*\uline[1]{
  \vbox{                       
       \hbox{\kern0em
      \ifmmode#1\else\ensuremath{#1}\fi%
      \kern0em                        
    }\kern-0.2ex%
    \hrule height 0.5pt
    \kern-0.3ex
    }}    

\begin{document}

\begin{center}
{\bf \Large Stable Calabi--Yau dimension of  self-injective algebras of finite type}

\vspace{30pt}

S. O. IVANOV and Y. V. VOLKOV \let\thefootnote\relax\footnote{The authors were supported by RFBR (grants no. 13-01-00902 A; 10-01-00635) and by SPbSU (NIR 6.38.74.2011). The first author acknowledges the support of RFBR (grant no. 12-01-31100 mol\_a) and ``Rokhlin grant''. }

\vspace{30pt}

\begin{minipage}{300pt} \footnotesize
ABSTRACT. We give an equivalent definition of the stable Calabi--Yau dimension in terms of bimodule syzygies and so-called stably inner automorphisms. Using it, we complete the computation of the stable Calabi--Yau dimensions of the self-injective algebras of finite representation type which was started by K.~Erdmann, A.~Skowro\'nski, J.~Bia\l kowski and A.~Dugas.
\end{minipage}
\end{center}

\vspace{20pt}

\section*{Introduction.}

Let $k$ be a field and $\mathcal{T}$ be a $k$-linear  hom-finite triangulated category with the shift functor  $\Sigma$. The category $\mathcal{T}$ is called a weakly $n$-Calabi--Yau category if there is a natural isomorphism
\begin{equation}\label{CY-def}{\rm Hom}_{\mathcal{T}}(y,\Sigma^nx)\cong D {\rm Hom}_{\mathcal{T}}(x,y),\end{equation}
 where $D={\rm Hom}_k(-,k).$ The weak Calabi--Yau dimension of  $\mathcal{T}$ is the least number  $n\geq 0$ such that $\mathcal{T}$ is a weakly $n$-Calabi--Yau category. This notion was introduced by M.~Kontsevich \cite{Konts}. It allows to interpret the Calabi--Yau property of a variety $X$ in terms of the derived category   $\mathcal{D}^b({\rm coh}(X)).$  In  \cite{Keller} B.~Keller introduced a modified definition of an $n$-Calabi--Yau category where the isomorphism (\ref{CY-def}) is compatible in a tricky way with the structure of a triangulated category. We call such a category strong $n$-Calabi--Yau category. In this article we deal only with weakly $n$-Calabi--Yau categories. 

In \cite{Erd_Skow} K.~Erdmann and A.~Skowro\'nski  introduced the notion of stable Calabi--Yau dimension of a self-injective algebra. It is just the (weak) Calabi--Yau dimension of the stable module category ${\rm \underline{mod}\text{-}} A.$

In \cite{Asashiba} H.~Asashiba classified all self-injective  algebras of finite representation type over algebraically closed field modulo derived equivalence. He proved that standard self-injective algebras of finite representation type are determined up to derived equivalence by their type $(T, f, t)$, where $T$ is the tree class, $f$ is the frequency and $t$ is the torsion order. Moreover, it was proved that nonstandard algebras of finite representation type occur only over a field of characteristic 2, and in this case there is only one series of such algebras up to derived equivalence. The following types occur in the case of standard self-injective algebra:
\begin{enumerate}
\item $(A_n, \frac{r}{n}, 1)$, $n\geq 1$, $r\geq 1$;
\item $(A_{2n+1}, r, 2)$, $n\geq 1$, $r\geq 1$;
\item  $(D_n, r, 1)$, $n\geq 4$, $r\geq 1$;
\item $(D_{3n}, \frac{r}{3}, 1)$, $n\geq 2$, $r\geq 1$;
\item $(D_n, r, 2)$, $n\geq 4$, $r\geq 1$;
\item $(D_4, r, 3)$, $r\geq 1$;
\item $(E_n, r, 1)$, $6\leq n\leq 8$, $r\geq 1$;
\item $(E_6, r, 2)$, $r\geq 1$;
\end{enumerate}
Only the following types occur in the case of  nonstandard algebras:
\begin{enumerate}
\setcounter{enumi}{8}
\item $(D_{3n}, \frac{1}{3}, 1)$, $n\geq 2$.
\end{enumerate}

The stable Calabi--Yau dimension was computed by K.~Erdmann, A.~Skowro\'nski, J.~Bia\l kowski and A.~Dugas in \cite{Bialkowski_Skowronski}, \cite{Erd_Skow}, \cite{Dugas} for all standard algebras except the algebras from the case 2 for even $n$ and the case 5 for even $r$. All these results are collected in  \cite{Dugas}.

The goal of this article is to reformulate the notion of stable Calabi--Yau dimension in terms of bimodule syzygies $\Omega^{n+1}_{A^e}(A)$, to complete the calculations for the remaining cases of standard algebras, and to compute the stable Calabi--Yau dimension of nonstandard algebras.

The structure of the paper is the following. In the beginning of the first section we give a definition of  {\it stably inner automorphism} of a self-injective algebra $A$. Roughly speaking, an automorphism is stably inner if it acts trivially on the stable module category.  The set of such automorphisms is denoted by  $\underline{\rm Inn}(A)$. It is a normal subgroup of the automorphism group of $A$ and it contains all inner automorphisms  ${\rm Inn}(A)\triangleleft \underline{\rm Inn}(A)\triangleleft {\rm Aut}(A)$. 
Further, we prove the following theorem.

\vspace{10pt}

\noindent{\bf Theorem \ref{thm_scydim}.}{\it\  Let  $A$ be a self-injective algebra without semisimple blocks. Then the following conditions are equivalent.
\begin{enumerate}
\item ${\rm \underline{mod}\text{-}} A$ is a weakly $n$-Calabi--Yau category.
\item $\Omega^{n+1}_{A^e}(A)\cong (A^{\vee})_\varphi$ for some $\varphi\in \underline{\rm Inn}(A).$ 
\item $\Omega^{-n-1}_{A^e}(A)\cong D(A)_\varphi$ for some $\varphi\in \underline{\rm Inn}(A).$
\end{enumerate}}
\noindent The second section is devoted to study of stably inner automorphisms. In particular, we prove the following property.

\ \\
{\bf Corollary  \ref{cor_st_in}.} {\it \ Let $A$ be a self-injective algebra without semisimple blocks, let $I$ be a two-sided ideal of $A$ and $\varphi\in\underline{\rm Inn}(A).$ Then  $\varphi(I)=I,$ and  $\varphi$ induces an inner automorphism of the quotient algebra $A/((0:I)+I).$}

\vspace{10pt}

\noindent Here we denote by $(0:I)$ the left annihilator of the ideal $I.$ The third section is devoted to the calculation of the stable Calabi--Yau dimension for the remaining cases of self-injective algebras of finite representation type: standard algebras of types $(A_{2n+1},r,2)$ for even $n$, $(D_n,r,2)$ for even $r$ and nonstandart algebras of type $(D_{3n},\frac{1}{3},1)$. The results of all the calculations are collected in Table 1. In this table we denote the greatest common divisor of numbers $a$ and $b$ by $(a,b)$ and by  $p$ the characteristic of the ground field $k.$

\begin{center}
Table 1:
\end{center}
\begin{supertabular}{|p{100pt}|p{70pt}|p{130pt}|}
\hline
Algebra type & Conditions & stable Calabi-Yau dimension\\
\hline
\hline
$(\Delta,\frac{r}{m_{\Delta}},1)$&
$(\frac{m_{\Delta}+1}{2},r)\ne 1$&$\infty$
\\
\hhline{~--}
$\Delta\in\{A_1,D_{2n},E_{7},E_{8}\}$ & $ 
 (\frac{m_{\Delta}+1}{2},r)=1,$  & $l$, where $0<l\leq r$ \\
 & $2\mid   r$ or $\ p=2$ & and $r\mid   (l-1)\frac{m_{\Delta}+1}{2}+1$
 \\
\hhline{~--}
&$(\frac{m_{\Delta}+1}{2},r)=1$,
 &
 $1+2l$, where $0\leq l< r$ 
 \\ & $2\nmid r$ and $p\ne 2$ &  and
 $r\mid  l(m_{\Delta}+1)+1$
  \\
\hline
$(\Delta,\frac{r}{m_{\Delta}},1),$&$(m_{\Delta}+1,r)\ne 1$&$\infty$\\
\hhline{~--}
$\Delta\not\in\{A_1,D_{2n},E_{7},E_8\}$&$(m_{\Delta}+1,r)=1$& 
 $1+2l$, where $0\leq l< r$ 
 \\ &  & and
 $r\mid  l(m_{\Delta}+1)+1$
\\
\hline
$(A_{2n+1},r,2)$, $n\geq 1$&$(r+n+1,2r)\ne1$&$\infty$\\
\hhline{~--}
&$(r+n+1,2r)=1$& $l(2n+1)-1$, where  $0< l\leq 2r$
 \\ &  & and $2r\mid  l(r+n+1)-1$
 \\
\hline
$(D_{n},r,2)$, $2\nmid  r$&$(n-1,r)\ne1$&$\infty$\\
\hhline{~--}
&$(n-1,r)=1$& $2l$, where $0<l< r(2n-3)$ 
 \\
& &  and $r(2n-3)\mid  l(2n-2)-(n-2)$
 \\
\hline
$(D_{n},r,2)$, $2\mid  r$& $(n-1,r)\ne1$ &$\infty$\\&or $p\ne 2$ &  
 \\
\hhline{~--}
& $(n-1,r)=1$ & $l(2n-3)-1$, where $0< l\leq 2r$ 
 \\ &and $p=2$ & and $2r\mid  l(n-1)-1$
 \\
\hline
$(D_4,r,3)$&---&$\infty$\\
\hline
$(D_{3n},\frac{1}{3},1), $&---&$4n-3$\\nonstandard & &  \\
\hline
$(E_6,r,2)$&$(6,r)\ne1$&$\infty$\\
\hhline{~--}
&$(6,r)=1$& $2l$, where $0< l< 11r$ 
\\& & and $11r\mid   12l-5 $
\\
\hline
\end{supertabular}

\section*{Acknowledgements.}

We wish to express our sincere gratitude to A. I. Generalov, A.O.~Zvonareva, A.A.~Ivanov, N.A.~Vavilov and A.~Dugas for instructive discussions and helpful remarks.

\section{Stable Calabi-Yau dimension and stably inner automorphisms.}

We fix a ground field $k$. Throughout this paper, all algebras will be unital, associative, finite-dimensional self-injective $k$-algebras, and all modules and bimodules will be finitely generated. Unless otherwise stated, modules are assumed to be right modules. For an algebra $A$ we denote by ${\rm mod\text{-}} A$ the category of finitely generated right $A$-modules, by $A{\text{-}\rm mod}$ the category of  finitely generated left $A$-modules, and by ${\rm bimod}\text{-} A$ the category of  finitely generated $A$-bimodules.

\subsection{Stably inner automorphisms.}

For an algebra homomorphism $\varphi : A\to B$ we denote by ${\rm res}_{\varphi}:{\rm mod\text{-}} B\to {\rm mod\text{-}} A$ the restriction functor, i.e. a functor $M\mapsto M_{\varphi}$  replacing the structure of $B$-module by the structure of $A$-module by the formula $m*a=m\varphi(a).$ It is clear that 
${\rm res}_\varphi\cong -\otimes_B B_\varphi.$ 

Let us remind that an automorphism $\varphi$ of an algebra  $A$ is called inner if there exists $a\in A$ such that  $\varphi(x)=a^{-1}xa$ for any $x\in A.$ A group of all inner automorphisms of $A$ is denoted by ${\rm Inn}(A).$ It is a normal subgroup of the group ${\rm Aut}(A)$ of all automorphisms.  The following proposition is proved in  \cite[Part I, Chapter II, proposition 5.3]{Bass}

\begin{proposition}  Let $\varphi$ be an automorphism of an algebra $A.$ The automorphism  $\varphi$ is inner if and only if ${\rm res}_\varphi\cong {\rm Id}.$
\end{proposition}

Let $A$ be a self-injective algebra and let $M,N$ be  $A$-modules. We denote by $\mathcal{P}(M,N)$ the subset 
of ${\rm Hom}_A(M,N)$ consisting of all homomorphisms passing through a projective module $M\to P\to N.$  The sets $\mathcal{P}(M,N)$ form an ideal of the category  ${\rm mod\text{-}} A.$ The stable module category ${\rm \underline{mod}\text{-}} A$ is a quotient category of ${\rm mod\text{-}} A/\mathcal{P}.$ It is a triangulated category with a shift functor given by Heller's cosyzygy functor $\Omega^{-1}.$ The set of morphisms from a module  $M$  to a module $N$ in the category ${\rm \underline{mod}\text{-}} A$ is denoted by  $\underline{\rm Hom}_A(M,N).$ It is clear  that if we fix an epimorphism from a projective module $\sigma: P\twoheadrightarrow N,$ a homomorphism lies in  $\mathcal{P}(M,N)$ if and only if it can be presented in the form $M\to P \overset{\sigma}{\twoheadrightarrow} N.$ 

For a morphism  $f$ in the category ${\rm mod\text{-}} A,$ we denote by $\underline{f}$ the corresponding morphism in the category ${\rm \underline{mod}\text{-}} A$. It is well-known that for  modules $M$ and $N$ without projective summands a homomorphism $f:M\to N$ is an isomorphism if and only if $\underline{f}:M\to N$ is an isomorphism in the category ${\rm \underline{mod}\text{-}} A.$

Let  $F:{\rm mod\text{-}} A \to {\rm mod\text{-}} B$ be a functor sending projectives to projectives. Then it induces a functor between stable categories which we denote by  $\underline{F}: {\rm \underline{mod}\text{-}} A\to {\rm \underline{mod}\text{-}} B$. 

\begin{definition}
An automorphism $\varphi$ of a self-injective algebra $A$ is called  {\it stably inner} if $\underline{\rm res}_\varphi \cong {\rm Id}_{{\rm \underline{mod}\text{-}} A}$. 
\end{definition} 

We denote by $\underline{\rm Inn}(A)$ a set of all stably inner automorphisms of an algebra $A.$ It is clear that  $\underline{\rm Inn}(A)$ is a normal subgroup of ${\rm Aut}(A).$

We stated the definition using right modules, but it is not essential because the following diagram is commutative
$$\xymatrix{
{\rm \underline{mod}\text{-}} A\ar[rr]^{\underline{D}} \ar[d]^{\underline{\rm res}_\varphi} & & A {\rm\text{-} \underline{mod}}\ar[d]^{{}_\varphi\underline{\rm res}} \\
{\rm \underline{mod}\text{-}} A\ar[rr]^{\underline{D}} & & A {\rm\text{-} \underline{mod}
} },
$$
where $D={\rm Hom}_{ k}(-, k).$

\subsection{Nakayama functor.}

For an algebra $A,$ we consider a functor
$$(-)^t={\rm Hom}_A(-,A):{\rm mod\text{-}} A \to A\text{-}{\rm mod}.$$
Recall that the functor $$\nu=D((-)^t):{\rm mod\text{-}} A\to {\rm mod\text{-}} A$$ is called Nakayama functor. It is an autoequivalence if and only if the algebra $A$ is self-injective, and it is isomorphic to the identity functor if and only if $A$ is a symmetric algebra.

Now we construct a contravariant endofunctor $(-)^{\vee}$  on the category of $A$-bimodules. One can consider an $A$-bimodule as a right module over the enveloping algebra $A^e=A^{\rm op}\otimes A$, and as a left module over $A^e.$ Therefore, we have isomorphisms of the categories  ${\rm bimod}\text{-} A\cong {\rm mod\text{-}} A^e$ and ${\rm bimod}\text{-} A \cong A^e\text{-}{\rm mod}.$ If we compose these isomorphisms with the functor $(-)^t:{\rm mod\text{-}} A^e \to A^e\text{-}{\rm mod},$ we obtain the functor 
$$(-)^{\vee}={\rm Hom}_{A^e}(-,A\otimes A):{\rm bimod}\text{-} A\to {\rm bimod}\text{-} A.$$

In other worlds, the space $A\otimes A$ is equipped by two commuting  $A$-bimodule structures, the {\it outer structure}: $$a\cdot (x_1\otimes x_2)\cdot b= (x_1\otimes x_2)(a\otimes b)= ax_1\otimes x_2b,$$ and the {\it inner} structure:
$$a\star (x_1\otimes x_2)\star b=(b\otimes a)(x_1\otimes x_2)=x_1b\otimes ax_2.$$ 
When we consider $A$-bimodule homomorphisms  $f:M\to A\otimes A$, we mean the outer bimodule structure on $A\otimes A$, but when we endow the space ${\rm Hom}_{A^e}(M,A\otimes A)$ by an  bimodule structure we use the inner bimodule structure   on  $A\otimes A.$

The first of the following isomorphisms is well-known and the second was proved in \cite{I_nak_funct_eilenberg_watts}.

\begin{proposition} \label{prop_nakayama} The following natural isomorphisms hold
 $$\nu\cong -\otimes_A D(A), \hspace{30pt} \nu^{-1} \cong -\otimes_A A^{\vee}.$$
\end{proposition}

An isomorphism of functors  \hbox{$-\otimes_A M$} $\cong$ \hbox{$-\otimes_A N$} includes an isomorphism of bimodules $M\cong N.$ Hence, we obtain 

\begin{corollary}\label{rem_A^vee_D(A)} The following isomorphisms of bimodules hold.
$$A^{\vee}\otimes_AD(A)\cong A, \hspace{30pt} D(A)\otimes_AA^{\vee}\cong A.$$
\end{corollary}

An algebra $A$ is called a Frobenius algebra if there exists a linear map \hbox{$\varepsilon:A\to  k$} such that the bilinear form $(a,b)\mapsto \varepsilon(ab)$ is nondegenerate. In this case  $\varepsilon$ is called a Frobenius form. Any Frobenius algebra is self-injective. The Nakayama automorphism of a Frobenius algebra $A$ is an automorphism $\tilde\nu :A\to A,$ defined by the equality $\varepsilon(a\cdot b)=\varepsilon(b\cdot \tilde\nu(a)),$ for any $a,b\in A.$  It is easy to check that  $D(A)\cong A_{\tilde\nu}$, and, consequently, \hbox{$\nu\cong {\rm res}_{\tilde\nu}$.} 
Using the last corollary, we get an isomorphism  $A^{\vee}\cong A_{\tilde\nu^{-1}}.$ Hence, we obtain
\begin{corollary}\label{cor_frobenious_iso} Let $A$ be a Frobenius algebra, and let $\tilde\nu$ be its Nakayama automorphism. Then the following isomorphisms of bimodules hold
$$D(A)\cong A_{\tilde\nu},\hspace{30pt}  A^{\vee}\cong A_{\tilde\nu^{-1}}.$$
\end{corollary}

\subsection{Stable Calabi-Yau dimension.}

Following K. Erdmann and A. Skowro\'nski  \cite{Erd_Skow}, we 
define the stable Calabi-Yau dimension of $A$ to be the weak Calabi-Yau dimension of the stable module category ${\rm \underline{mod}\text{-}} A$.  They proved that the functor $\Omega\text{\b{$\nu$}}$ is a Serre functor of the category  ${\rm \underline{mod}\text{-}} A.$ Therefore, the category   ${\rm \underline{mod}\text{-}} A$ is a weakly $n$-Calabi-Yau category if and only if there is an isomorphism of functors $\Omega^{n+1}\cong \text{\b{$\nu$}}^{-1}.$ Thus, the stable Calabi-Yau dimension of an algebra  $A$ is the least number  $n\geq 0$ such that  $\Omega^{n+1}\cong\text{\b{$\nu$}}^{-1}.$

\begin{remark}
The category ${\rm \underline{mod}\text{-}} A$ is a weakly  $n$-Calabi-Yau category if and only if the category  $A\text{-}\underline{\rm mod}$ is a weakly  $n$-Calabi-Yau category. It follows from the commutativity of the following diagrams.
$$\xymatrix{
{\rm \underline{mod}\text{-}} A\ar[rr]^{\underline{D}}\ar[d]^{\Omega^m} & & A\text{-}\underline{\rm mod}\ar[d]^{\Omega^{-m}} \\
{\rm \underline{mod}\text{-}} A\ar[rr]^{\underline{D}} & & A\text{-}\underline{\rm mod}
}\hspace{30pt} \xymatrix{
{\rm \underline{mod}\text{-}} A\ar[rr]^{\underline{D}}\ar[d]^{\underline{\nu}} & & A\text{-}\underline{\rm mod}\ar[d]^{\underline{\nu}^{-1}} \\
{\rm \underline{mod}\text{-}} A\ar[rr]^{\underline{D}} & & A\text{-}\underline{\rm mod}
}$$ 
\end{remark}

Let $_AM_B$ be a bimodule which is left $A$-projective and right $B$-projective. Then the functor \hbox{$-\otimes_AM:$} ${\rm mod\text{-}} A\to {\rm mod\text{-}} B$ is an exact functor and maps projectives to projectives. Thus, it induces a functor between stable categories, we denote it by  $-\otimes^{\bf st}_AM:{\rm \underline{mod}\text{-}} A\to {\rm \underline{mod}\text{-}} B$. For example, we have the following isomorphisms
$$\underline{\rm res}_\varphi\cong -\otimes_A^{\bf st} A_\varphi, \hspace{10pt} \Omega^m\cong -\otimes_A^{\bf st} \Omega^m_{A^e}(A),$$
$$\underline\nu\cong -\otimes^{\bf st}_A D(A),\hspace{10pt}  \underline\nu^{-1}\cong -\otimes^{\bf st}_A A^{\vee}.$$ 

For an $A$-bimodule $M$ we denote by $M_A$ this bimodule considered as a right module, and by ${}_AM$ this bimodule considered as a left module.

\begin{lemma}\label{lem_A_f}
Let $A$ be an algebra with the radical $J$ and  $M$ be a right-left projective $A$-bimodule. Then the following conditions are equivalent. 
\begin{enumerate}
\item  $(M/JM)_A\cong (A/J)_A$ and  ${}_A(M/MJ)\cong {}_A(A/J)$.
\item  ${}_A(M/JM)\cong {}_A(A/J)$ and $(M/MJ)_A\cong (A/J)_A$.
\item $M_A\cong A_A$ and ${}_AM\cong {}_AA.$
\item $M\cong A_\varphi$ (as bimodules) for some $\varphi\in {\rm Aut}(A).$
\end{enumerate}
\end{lemma}
\begin{proof} {\bf (1)$\Rightarrow$(2)} Since  $(M/JM)_A\cong (A/J)_A$, we obtain that the right module $(M/JM)_A$ is semisimple. Hence,  ${\rm rad}(M_A)\subseteq JM.$ In other words, we have $MJ\subseteq JM.$ Similarly, we have $JM\subseteq MJ.$ Thus, we obtain $JM=MJ$. It follows that (2) holds.

{\bf (2)$\Rightarrow$(3)} Since $M_A$ and $A_A$ are projective modules with isomorphic tops, we obtain  $M_A\cong A_A$. Similarly, we get ${}_AM\cong {}_AA.$

{\bf (3)$\Rightarrow$(4)} Let us denote an isomorphism of the left modules by $\tau: {}_AA \to {}_AM.$ Consider the bimodule $\widetilde A$ such that ${}_A(\widetilde A)={}_AA$ and the right module structure is given by  $x*a=\tau^{-1}(\tau(x)a).$ Then $\tau:\widetilde A\to M$ is a bimodule isomorphism.

We put $\varphi(a):=1*a.$ Then  $\varphi:A\to A$ is a linear map, $\varphi(1)=1,$ and $$\varphi(ab)=1*(ab)=(1*a)*b=\varphi(a)*b=$$
$$=(\varphi(a)\cdot 1)*b=\varphi(a)\cdot (1*b)=\varphi(a)\varphi(b).$$ Hence, $\varphi$ is an endomorphism of the algebra $A.$ Since $\tilde A_A\cong M_A \cong A_A,$ we obtain ${\rm Ker}(\varphi)={\rm Ann}(\widetilde A_A)={\rm Ann}(A_A)=0$. Consequently, $\varphi\in {\rm Aut}(A).$ From the equality  $x*a=(x\cdot 1)\star a=x\cdot (1\star a)=x\varphi(a),$ it follows that  $\widetilde A=A_\varphi.$ Hence, we have $M\cong A_\varphi$. 

{\bf (4)${\Rightarrow}$(3)} The linear map  $\varphi:A_A\to (A_\varphi)_A$ is a right module isomorphism. Thus, $(A_\varphi)_A\cong A_A$ and  ${}_A(A_\varphi)={}_AA.$

{\bf (3)${\Rightarrow}$(1)} The proof is obvious.
\end{proof}

An algebra $A$ is said to be an algebra without  semisimple blocks if $A$ can not be presented in a form  $A=A_1\times A_2,$ where $A_1$ is a semisimple algebra. This condition is equivalent to the following condition: any simple $A$-module is  non-projective or non-injective. For a self-injective algebra this is  equivalent to the requirement that any semisimple module is not projective. In particular, it follows that two semisimple modules are isomorphic if and only if they are isomorphic in the stable module category. 

\vbox{
\begin{theorem}\label{thm_scydim}
Let  $A$ be a self-injective algebra without semisimple blocks. Then the following conditions are equivalent.
\begin{enumerate}
\item ${\rm \underline{mod}\text{-}} A$ is a weakly $n$-Calabi-Yau category.
\item $\Omega^{n+1}_{A^e}(A)\cong (A^{\vee})_\varphi$ for some $\varphi\in \underline{\rm Inn}(A).$ 
\item $\Omega^{-n-1}_{A^e}(A)\cong D(A)_\varphi$ for some $\varphi\in \underline{\rm Inn}(A).$
\end{enumerate}
\end{theorem}}

\begin{proof}
The implications (2)$\Rightarrow$(1) and (3)$\Rightarrow$(1) follow from the isomorphisms  $-\otimes^{\bf st}_A\Omega_{A^e}^{n+1}(A)\cong \Omega^{n+1},$\ $-\otimes^{\bf st}_AA^{\vee}_\varphi \cong \underline{\nu}^{-1}$ and $-\otimes^{\bf st}_A\Omega_{A^e}^{-n-1}(A)\cong \Omega^{-n-1},$ \  $-\otimes^{\bf st}_AD(A)_\varphi \cong \underline{\nu}.$ Let us prove  (1)$\Rightarrow$(2)$\wedge$(3). 

Let ${\rm \underline{mod}\text{-}} A$ be a weakly $n$-Calabi-Yau category.  Consider a minimal complete projective resolution  ${\bf P}$ of the bimodule $A.$ 
$$\xymatrix{
\dots \ar[r]^{d^{\bf P}_2} & {\bf P}_{1} \ar[r]^{d^{\bf P}_1} & {\bf P}_0 \ar[rr]^{d^{\bf P}_0}\ar@{->>}[rd] && {\bf P}_{-1}\ar[r]^{d^{\bf P}_{-1}} & {\bf P}_{-2} \ar[r]^{d^{\bf P}_{-2}}& \dots\\
& & & A\  \ar@{>->}[ur] & & & }$$

By definition, we have $\Omega^m_{A^e}(A)={\rm Ker}(d_m^{\bf P})$ for any $m\in \mathbb{Z}.$ All ${\bf P}_m$ are projective bimodules, so ${}_A({\bf P}_m)$ are projective left modules. The bimodule ${}_AA$ is a projective left module too. Hence, we have that the complex ${}_A{\bf P}$ is a contractible complex of left modules. Therefore, for any $A$-module $M$ the complex $M\otimes_A {\bf P}$ is a complete projective resolution of the module $M$ and $M\otimes_A\Omega^m_{A^e}(A)={\rm Ker}(d_m^{M\otimes_A{\bf P}}).$ The minimality of the complete resolution ${\bf P}$ is equivalent to the inclusion  ${\rm Im}(d_m^{{\bf P}})\subseteq {\rm rad}({\bf P}_{m-1})$ for all $m\in \mathbb{Z}.$  Consider an arbitrary right semisimple module $S.$ Then  ${\rm rad}(S\otimes_A {\bf P}_m)=S\otimes_A {\rm rad} ({\bf P}_m)$ for any integer $m,$ and, consequently, ${\rm Im}(1_S\otimes_A d^{\bf P}_m) \subseteq {\rm rad}(S\otimes_A {\bf P}_{m-1}).$ Thus, the complex  $S\otimes_A {\bf P}$ is a minimal complete projective resolution of $S.$ Hence, modules  ${\rm Ker}(d_m^{S\otimes_A{\bf P}})=S\otimes_A \Omega^m_{A^e}(A)$   have no projective summands. The modules  $\nu(S)$ and $\nu^{-1}(S)$ are semisimple too. Since the algebra $A$ has no semisimple blocks, the isomorphisms of functors 
 $\Omega^{n+1}\cong \underline{\nu}^{-1}$ and $\Omega^{-n-1}\cong \underline{\nu}$ include isomorphisms of the right modules $S\otimes_A \Omega^{n+1}_{A^e}(A)\cong \nu^{-1}(S)$ and $S\otimes_A \Omega^{-n-1}_{A^e}(A)\cong \nu(S)$ for any semisimple right module $S.$  Similarly, using the fact that the stable module category of left modules is a weakly $n$-Calabi-Yau category too, we obtain the  isomorphisms of left modules  $ \Omega^{n+1}_{A^e}(A)\otimes_AT\cong \nu^{-1}(T)$ and $ \Omega^{-n-1}_{A^e}(A)\otimes_AT\cong \nu(T)$ for any semisimple left module $T.$

Let us denote $M_1=D(A)\otimes_A \Omega^{n+1}_{A^e}(A)$ and $M_{2}=A^{\vee}\otimes_A \Omega^{-n-1}_{A^e}(A).$  Combining the obtained isomorphisms, the proposition \ref{prop_nakayama}, and its analogue for left modules,  we get that $S\otimes_A M_i\cong S$ for any semisimple right module  $S,$ and $ M_i\otimes_A T\cong T$ for any semisimple left module $T(i=1,2).$ Denote by $J$ the radical of the algebra $A.$ Therefore, we have  ${}_A(M_i/JM_i)\cong {}_A((A/J)\otimes_AM_i)\cong {}_A(A/J)$, and $(M_i/M_iJ)_A\cong (M_i\otimes_A(A/J))_A\cong (A/J)_A$. From the lemma \ref{lem_A_f} it follows that  $M_i\cong A_{\varphi_i}$ for some $\varphi_i\in {\rm Aut}(A).$ 
It is clear that $\underline{\rm res}_{\varphi_1}\cong -\otimes_A^{\bf st} M_1\cong \underline{\nu}\circ \underline{\nu}^{-1} \cong {\rm Id}$ and  $\underline{\rm res}_{\varphi_2}\cong -\otimes_A^{\bf st} M_2\cong \underline{\nu}^{-1}\circ \underline{\nu} \cong {\rm Id}.$
 Thus, we have two isomorphisms $D(A)\otimes_A \Omega^{n+1}_{A^e}(A)\cong A_{\varphi_1}$ and $A^{\vee}\otimes_A \Omega^{-n-1}_{A^e}(A)\cong A_{\varphi_2},$ where $\varphi_1,\varphi_2\in \underline{\rm Inn}(A).$ If we apply the functor  $A^{\vee}\otimes_A -$ to the first isomorphism, apply the functor $D(A)\otimes_A-$ to the second isomorphism, and use the corollary \ref{rem_A^vee_D(A)}, we obtain  $\Omega^{n+1}_{A^e}(A)\cong (A^{\vee})_{\varphi_1}$ and  $\Omega^{-n-1}_{A^e}(A)\cong D(A)_{\varphi_2},$ where $\varphi_1,\varphi_2\in \underline{\rm Inn}(A).$
\end{proof}

\begin{corollary} \label{cor_scydim}
Let $A$ be a Frobenius algebra without semisimple blocks with a Nakayama automorphism $\tilde\nu$. Then the following conditions are equivalent.
\begin{enumerate}
\item ${\rm \underline{mod}\text{-}} A$ is a weakly $n$-Calabi-Yau category.
\item $\Omega^{n+1}_{A^e}(A)\cong A_{\tilde\nu^{-1}\varphi}$ for some $\varphi\in \underline{\rm Inn}(A).$ 
\end{enumerate}
\end{corollary}
\begin{proof}
It follows from the previous theorem and the corollary \ref{cor_frobenious_iso}.
\end{proof}

\section{Stably inner automorphisms.}

\subsection{Inner modulo socle automorphisms.}
In this subsection we introduce a notion of an inner modulo socle automorphism and prove that any inner modulo socle automorphism is a stably inner automorphism. It gives a lot of examples of stably inner automorphisms which are not inner.

It is well-known that   ${\rm soc}(A_A)={\rm soc}(_AA)$ if $A$ is a self-injective algebra. We denote briefly  ${\rm soc}(A):={\rm soc}(_AA)$ and $A/{\rm soc}:=A/{\rm soc}(A).$ Any automorphism  $\varphi$ of an algebra sends  ${\rm soc}(A)$ to itself. Thus, $\varphi$ induces an automorphism   $\varphi/{\rm soc} : A/{\rm soc}\to A/{\rm soc}.$

\begin{definition} 
An automorphism $\varphi$ of an algebra $A$ is called {\it inner modulo socle} if the induced automorphism $\varphi/{\rm soc}:A/{\rm soc}\to A/{\rm soc}$ is inner.
\end{definition}

\begin{proposition}
Let $A$ be a self-injective algebra and $\varphi\in {\rm Aut}(A)$. If  $\varphi$ is inner modulo socle, then it is stably inner.
\end{proposition}

\begin{proof}
Let us consider a functor ${\rm mod\text{-}} (A/{\rm soc}) \to {\rm \underline{mod}\text{-}} A$ obtained by composition of the restriction functor ${\rm mod\text{-}} (A/{\rm soc}) \to {\rm mod\text{-}} A$  and the canonical projection ${\rm mod\text{-}} A\to {\rm \underline{mod}\text{-}} A.$ Denote the image of this functor by ${\underline{\mathcal{M}}}.$ It is well-known that a module over a self-injective algebra has no projective summands if and only if its annihilator includes ${\rm soc}(A).$ Thus,   ${\underline{\mathcal{M}}}$ is a full subcategory consisting of modules without projective summands, and it is equivalent to the whole category ${\rm \underline{mod}\text{-}} A.$ Further, the following  diagram is commutative. 
$$\xymatrix{
{\rm mod\text{-}} (A/{\rm soc})\ar@{->>}[rr]\ar[d]^{{\rm res}_{\varphi/{\rm soc}}} & & {\underline{\mathcal{M}}}\ar[d]^{\underline{\rm res}_{\varphi}}\ar@{^(->}@<-2pt>[rr]^{\simeq} & & {\rm \underline{mod}\text{-}} A\ar[d]^{\underline{\rm res}_{\varphi}} \\
{\rm mod\text{-}} (A/{\rm soc})\ar@{->>}[rr] & & {\underline{\mathcal{M}}} \ar@{^(->}@<-2pt>[rr]^{\simeq}  & &  {\rm \underline{mod}\text{-}} A}$$

If  $\varphi$ is inner modulo socle, then ${\rm res}_{\varphi/{\rm soc}}\cong {\rm Id}.$ It follows that the functor $\underline{\rm res}_{\varphi}$ restricted on ${\underline{\mathcal{M}}}$ is isomorphic to the identity functor too. From the commutativity of the right square we get that $\underline{\rm res}_{\varphi}$ is isomorphic to the identity functor on the whole category ${\rm \underline{mod}\text{-}} A.$
\end{proof}

\subsection{Action of a stably inner automorphism on modules.}

It is well-known that two modules without projective summands are isomorphic in the stable module category if and only if they are isomorphic as modules. It follows that for a stably inner automorphism $\varphi$ and a module $M$ without projective summands there is an isomorphism  $M\cong M_{\varphi}.$ In this subsection we prove that this isomorphism holds for any module. At first, we prove a technical lemma which we need further.

For any $A$-module $M$ we choose submodules  $M_{\mathcal{P}}$ and $P_M$ such that  $M=M_{\mathcal{P}}\oplus P_M$, the module $M_{\mathcal{P}}$ has no projective summands, and the  module $P_M$ is projective. For  a homomorphism $f:M\to N$ we denote by  $f_{\mathcal{P}}:M_{\mathcal{P}}\to N_{\mathcal{P}}$ the composition of $f$ with the inclusion  $M_{\mathcal{P}}\to M$ and the projection $N\to N_{\mathcal{P}}.$ For  homomorphisms  $f,g:M\to N$  the equality  $f_{\mathcal{P}}=g_{\mathcal{P}}$  includes  $\underline{f}=\underline{g}.$ It is well-known that for a homomorphism  $f:M\to N$ the morphism $\underline{f}$ is an isomorphism if and only if  $f_{\mathcal{P}}$ is an isomorphism   (for example, see the proof of the  lemma 8.1 in \cite{AR_St_I}).

\begin{lemma}\label{horoshiy_vibor_isomorphisma}
Let $\varphi$ be an automorphism of an algebra $A$ such that for any simple module $S$ there is an isomorphism  $S\cong S_{\varphi}$. Moreover, let $M$ be a module and  $\Phi:M\to M_{\varphi}$ be an isomorphism in the stable module category ${\rm \underline{mod}\text{-}} A.$ Then there is an isomorphism  $f:M\to M_{\varphi}$ in the category  ${\rm mod\text{-}} A$ such that $\underline{f}\cong \Phi.$
\end{lemma}
\begin{proof}
Since $S\cong S_\varphi$ for any simple module $S,$ we get ${\rm top}(P_{\varphi})={\rm top}(P)_{\varphi}\cong {\rm top}(P)$ for any projective module $P$. Using that  $P_\varphi$ is a projective module, we obtain $P\cong P_{\varphi}$ for any projective module $P.$

Let us choose a homomorphism $\tilde f:M\to M_{\varphi}$ such that $\underline{\tilde f}=\Phi.$ 
It is clear that $(\tilde f)_{\mathcal{P}}$ is an isomorphism.  Since  ${\rm res}_{\varphi}$ is an additive functor, we get $M_{\varphi}\cong (M_{\mathcal{P}})_{\varphi}\oplus (P_{M})_{\varphi}\cong (M_{\mathcal{P}})_{\varphi}\oplus P_{M} .$ Using the Krull-Schmidt theorem, we obtain $P_{M}\cong P_{M_{\varphi}}.$ Denote this isomorphism by $\theta:P_{M}\to P_{M_{\varphi}}.$ Then we present modules as a direct sum $M=M_{\mathcal{P}}\oplus P_{M}$, $M_{\varphi}=(M_{\varphi})_{\mathcal{P}}\oplus P_{M_{\varphi}}$ and define an isomorphism  $f:M\to M_{\varphi}$ by the formula $f=(\tilde f)_{\mathcal{P}}\oplus \theta_M.$  Finally, since $( f)_{\mathcal{P}}=(\tilde f)_{\mathcal{P}},$ we obtain  $\underline{f}=\underline{\tilde f}=\Phi.$
\end{proof}

\begin{proposition}
Let $A$ be a self-injective algebra without semisimple blocks, and $\varphi$ is a stably inner automorphism of $A.$ Then for any $A$-module $M$ there is an (not necessarily natural) isomorphism
$M\cong M_{\varphi}.$
\end{proposition}

\begin{proof}
Since the algebra $A$ has no semisimple blocks, any simple $A$-module is non-projective. Thus, the isomorphism $S\cong S_\varphi$ in the category ${\rm \underline{mod}\text{-}} A$ includes the isomorphism  $S\cong S_\varphi$ for any simple module $S.$ Hence, we can use the previous lemma and lift the isomorphisms $\Phi_M:M\to M_\varphi$ from the stable module category to the category of modules. 
\end{proof}

\subsection{Stable cyclic module category.}

Let $A$ be a finite dimensional algebra. Denote by ${\rm cycl}\text{-}A$ the full subcategory of ${\rm mod\text{-}} A$ consisting of cyclic modules. Any cyclic module is isomorphic to a module of the form  $A/I,$ where $I$ is a right ideal.  Denote by ${\rm cycl}_0\text{-}A$ the full subcategory of  ${\rm cycl}\text{-}A$ consisting of modules of the form  $A/I.$ Then the inclusion  $\iota: {\rm cycl}_0\text{-}A \to {\rm cycl}\text{-}A$ is an equivalence. We denote by  $ \uline{\rm cycl}\text{-} A$ and $\uline{\rm cycl}_0\text{-} A$ the associated subcategories in the category ${\rm \underline{mod}\text{-}} A.$   The functor ${\rm res}_{\varphi}$ maps cyclic modules to cyclic modules. Denote by  ${\rm res}^{\rm cycl}_{\varphi}:{\rm cycl}\text{-}A \to {\rm cycl}\text{-}A$ its restriction. In this subsection we investigate these categories and the action of ${\rm res}^{\rm cycl}_\varphi$ on them.

For right ideals $I,J$ of $A$ we put 
$(J:I):=\{a\in A \mid  aI\subseteq J\}.$
Let us notice that $(J:I)$ is a vector subspace and   $J\subseteq (J:I).$

\begin{lemma}\label{cycl_hom} The map $f\mapsto f(1+I)$ gives the isomorphism:
$${\rm Hom}_A(A/I,A/J)\cong \frac{(J:I)}{J}.$$
\end{lemma}
\begin{proof} The proof is obvious. 
\end{proof}

For $c\in (J:I),$ we denote by $(c\, \cdot)$ the unique homomorphism $(c\,\cdot)=f:A/I\to A/J$ such that $f(1+I)=c+J.$ Using this notation, the composition is expressed in the following way
\hbox{$(c_1 \cdot )\circ (c_2\cdot)=(c_1c_2\cdot).$}

\begin{proposition}\label{cycl_sthom}
The map $f\mapsto f(1+I)$ induces an isomorphism:
$$\underline{{\rm Hom}}_A(A/I,A/J)\cong \frac{(J:I)}{(0:I)+J}$$
\end{proposition}

\begin{proof} 
It is sufficient to prove that the image of $\mathcal{P}(A/I,A/J)$ under the isomorphism  ${\rm Hom}_A(A/I,A/J)\cong \frac{(J:I)}{J}$  is equal to  $\frac{(0:I)+J}{J}.$ In other words, it is sufficient to prove that 
$f\in \mathcal{P}(A/I,A/J)$ if and only  if  $f=(c\:\cdot)$ where $c\in(0:I).$

We know that $f\in \mathcal{P}(A/I,A/J)$ if and only if $f$ is presented as a composition  $f=\sigma g,$ where $g:A/I\to A$ is a homomorphism and $\sigma : A \twoheadrightarrow A/J$ is the canonical projection. The canonical projection $\sigma$ can be written as  $\sigma=(1\cdot).$ Using the last lemma we obtain that any homomorphism $g:A/I\to A$ can be presented in the form $g=(c\:\cdot),$ where $c\in (0:I).$ Thus,  $f\in \mathcal{P}(A/I,A/J)$ if and only if $f$ can be presented in the form $f=(1\cdot)(c\:\cdot)=(c\:\cdot),$ where $c\in (0:I).$ 
\end{proof}

Let $\varphi$ be an automorphism of an algebra $A.$ We define the endofunctor $\mathfrak{r}_\varphi$ on the category ${\rm cycl}_0\text{-}A$ as follows: $\mathfrak{r}_\varphi (A/I)=A/\varphi^{-1}(I);$ $\mathfrak{r}_\varphi ((c\:\cdot))=(\varphi^{-1}(c)\cdot).$

\begin{lemma} Let $\varphi$ be an automorphism of an algebra $A.$ Then there is an isomorphism
${\rm res}^{\rm cycl}_\varphi\circ\iota\cong\iota\circ \mathfrak{r}_\varphi.$
$$\xymatrix{
{\rm cycl}_0\text{-}A\ar[rr]^\iota_{\simeq}\ar[d]^{\mathfrak{r}_\varphi} & & {\rm cycl}\text{-}A\ar[d]^{{\rm res}^{\rm cycl}_\varphi} \\
{\rm cycl}_0\text{-}A\ar[rr]^\iota_{\simeq} & & {\rm cycl}\text{-}A 
}$$
\end{lemma}

\begin{proof}
Let $I$ be a right ideal of the algebra $A$. Then the map  $a+\varphi^{-1}(I)\mapsto \varphi(a)+I$ is an isomorphism of the modules
$A/\varphi^{-1}(I) \cong (A/I)_{\varphi}.$
It is sufficient to verify that for any right ideals  $I,J$ and for any $c\in (J:I)$ the following diagram is commutative.
$$\xymatrix{
A/\varphi^{-1}(I)\ar[rr]^{(\varphi^{-1}(c)\cdot)}\ar[d]^{\cong} & & A/\varphi^{-1}(J) \ar[d]^{\cong}\\
(A/I)_{\varphi}\ar[rr]^{(c\:\cdot)}  && (A/J)_{\varphi}
}$$ We leave this verification to the reader.
\end{proof}

\subsection{Automorphisms acting trivially on the stable cyclic module category.}

If $\varphi$ is a stably inner automorphism, the restricted functor  $\underline{\rm res}^{\rm cycl}_\varphi$ is isomorphic to the identity functor. Thus, we have the following inclusions of groups.
\begin{center}
$ \left\{ \begin{matrix}
\text{inner modulo socle} \\
\text{automorphisms}  \end{matrix} \right\}\hspace{10pt} \subseteq\hspace{10pt} \underline{\rm Inn}(A) \hspace{10pt}\subseteq\hspace{10pt}
\left\{ \begin{matrix}
\text{automorphisms} 
 \\ { \varphi : \underline{\rm res}^{\rm cycl}_\varphi\cong {\rm Id}}  \end{matrix} \right\} $
\end{center}

We are interested in the middle group, but it is difficult to describe this group explicitly. The left group is clear. In the following theorem we describe the right group. It will give us some properties of stably inner automorphisms. 

\begin{theorem}\label{stable_inner_theorem} 
Let $A$ be a self-injective algebra without semisimple blocks and $\varphi\in {\rm Aut}(A).$ Then  $\underline{\rm res}^{\rm cycl}_\varphi\cong {\rm Id}$ if and only if there exists a family of invertable elements $\{\xi_I\}$ of the algebra $A,$ where $I$ runs over the the set of all  proper non-zero right ideals, such that the following properties hold. 
\begin{enumerate}
\item $\varphi(I)= \xi_II.$
\item  
$a\in (J:I) \ \Rightarrow \  \varphi(a)\xi_I-\xi_Ja\in (0:I)+\varphi(J),$  for all possible $I,J$.
\end{enumerate} Moreover, if we replace the family  $\{\xi_I\}$ by a family of invertible elements  $\{\xi'_I\}$ such that $\xi_I-\xi'_I\in  (0:I)+\varphi(I),$ then the properties 1 and 2 hold for this family. 
\end{theorem} 
\begin{remark}
For any inner automorphism $\varphi(a)=\xi a\xi^{-1},$ we may take $\xi_I:=\xi.$
\end{remark}
\begin{proof}[Proof of theorem \ref{stable_inner_theorem}]
The functor $\underline{\rm res}^{\rm cycl}_\varphi$ is isomorphic to the identity functor if and only if the functor $\underline{\mathfrak{r}}_{\varphi^{-1}}=\underline{\mathfrak{r}}^{-1}_{\varphi}\cong \underline{\rm res}^{\rm cycle}_{\varphi^{-1}}$ is isomorphic to the identity functor. Let  $f: {\rm Id}\to \underline{\mathfrak{r}}_{\varphi^{-1}}$ be a morphism of functors. Then its components  $f_{A/I}:A/I\to A/\varphi(I)$ can be written in the form  $f_{A/I}=\underline{(\xi_I\cdot)},$ where $\xi_I\in (\varphi(I):I).$ From the commutativity of the diagram 
$$\xymatrix{
A/I\ar[rr]^{(a\cdot)}\ar[d]^{(\xi_I\cdot)} & & A/J\ar[d]^{(\xi_J\cdot)}\\
A/\varphi(I)\ar[rr]^{(\varphi(a)\cdot)} & & A/\varphi(J)
}$$ for any $I,J,$  $a\in (J:I)$  and the proposition \ref{cycl_sthom}, it follows that $\varphi(a)\xi_I-\xi_Ja\in (0:I)+\varphi(J)$. Conversely, if conditions  $\xi_I\in (\varphi(I):I)$ and $\varphi(a)\xi_I-\xi_Ja\in (0:I)+\varphi(J)$ for any $a\in (J:I)$ hold, the family $\{\xi_I\}$ defines the morphism of functors $f: {\rm Id}\to \underline{\mathfrak{r}}_{\varphi^{-1}}.$ Moreover, if we replace the family  $\{\xi_I\}$ with a family  $\{\xi'_I\}$ such that $\xi_I-\xi'_I\in  (0:I)+\varphi(I),$ then  $f_{A/I}=\underline{(\xi_I\cdot)}=\underline{(\xi'_I\cdot)}$ and the family  $\{\xi'_I\}$ defines the same morphism of functors.

If the element $\xi_I$ is invertible, the condition $\xi_I\in (\varphi(I):I)$ is equivalent to   $\varphi(I)=\xi_II.$ It is sufficient to prove that a morphism of functors $f$ is an isomorphism if and only if we can choose invertible elements $\xi_I$ such that $f_{A/I}=\underline{(\xi_I\cdot)}$. 

If elements $\xi_I\in (\varphi(I):I)$ are invertible, then $\xi_II=\varphi(I)$, hence, $I=\xi_I^{-1}\varphi(I),$ and $\xi_I^{-1}\in (I:\varphi(I)).$ Thus, we get morphisms  $(\xi_I^{-1}\cdot):A/\varphi(I)\to A/I,$ which are inverse to morphisms  $f_{A/I}.$ Hence, $f$ is an isomorphism.

Let now $f$ be an isomorphism. From the lemma \ref{horoshiy_vibor_isomorphisma} it follows that we can choose $\{\tilde\xi_I\}$ such that  $f_{A/I}=\underline{(\tilde\xi_I\cdot)}$ and  that  $(\tilde\xi_I\cdot):A/I\to A/\varphi(I)$ is an isomorphism in the category ${\rm mod\text{-}} A.$ Using the lemma \ref{cycl_hom}, we obtain that there exist  $\lambda_I\in (I:\varphi(I))$ such that  $\tilde\xi_I\lambda_I-1\in \varphi(I).$ Hence, we get an equality  $\varphi(I)+\tilde\xi_IA=A.$ From  \cite[Ch III, 2.8]{Bass} it follows that there exist invertible elements  $\xi_I\in\varphi(I)+\tilde\xi_I$. Thus, we have found invertible elements  $\xi_I$ such that  $f_{A/I}=\underline{(\xi_I\cdot)}.$
\end{proof}

\begin{corollary}\label{cor_st_in}
Let $A$ be a self-injective algebra without semisimple blocks, and let $I$ be a two-sided ideal of $A$ and $\varphi\in\underline{\rm Inn}(A).$ Then  $\varphi(I)=I,$ and  $\varphi$ induces an inner automorphism on the quotient algebra $A/((0:I)+I).$
\end{corollary}
\begin{proof}
Since $\varphi$ is a stably inner automorphism, there is a family $\{\xi_I\}$ satisfying conditions of the previous theorem. For a two-sided ideal $I$, we have  $\varphi(I)=\xi_II=I.$ Hence, from the condition $\varphi(a)\xi_I-\xi_Ia\in (0:I)+I$ for any $a\in (I:I)=A,$ it follows that $\varphi$ induces an inner automorphism on $A/((0:I)+I)$. 
\end{proof}

We denote by  ${\rm soc}^i({}_AA)$ the $i$-th term of the socle series of the left regular module ${}_AA$.

\begin{corollary}\label{cor_st_inn_loops} Let $\mathcal{Q}=(\mathcal{Q}_0,\mathcal{Q}_1,s,t)$ be a quiver,  $A= k \mathcal{Q}/I$ be a self-injective bound quiver algebra  such that  ${\rm soc}^2({}_AA) \subseteq {\rm rad}^2(A),$ and $\varphi\in \underline{\rm Inn}(A).$ Then there exists a family $\{d_i\}_{i\in \mathcal{Q}_0}$ of non-zero elements in the ground field $k$ such that  $$\varphi(e_i)-e_i \in {\rm rad}(A) \hspace{10pt} \text{ and } \hspace{10pt} \varphi(\alpha) - \frac{d_{t(\alpha)}}{d_{s(\alpha)}}\alpha\in  {\rm rad}^2(A)$$  for all $i\in \mathcal{Q}_0, \alpha\in \mathcal{Q}_1.$  In particular, if $\alpha$ is a loop, then $\varphi(\alpha)-\alpha \in {\rm rad}^2(A).$
\end{corollary}
\begin{proof} The property ${\rm soc}^2({}_AA) \subseteq {\rm rad}^2(A)$ implies that the algebra $A$ has no semisimple blocks. 
Notice that  $(0:{\rm rad}^2(A))={\rm soc}^2({}_AA),$  and hence, $(0:{\rm rad}^2(A))+{\rm rad}^2(A)={\rm rad}^2(A).$ Using the previous corollary, we obtain that the automorphism $\varphi$ induces the inner automorphism $a\mapsto \xi a \xi^{-1}$  on the quotient algebra $A/{\rm rad}^2(A).$ The invertible element  $\xi$ has the form $\xi=\xi_0+r,$ where $\xi_0=\sum_{i\in \mathcal{Q}_0}d_ie_i,$ $d_i\in  k\setminus \{0\}$ and $r\in {\rm rad}(A).$ Then we get  $\varphi(\alpha)-\xi_0\alpha\xi_0^{-1} \in {\rm rad}^2(A)$ for any arrow $\alpha$. And since $e_i\xi_0=\xi_0e_i,$ we obtain $\varphi(e_i)-e_i\in{\rm rad}(A).$ This completes the proof.
\end{proof}

\subsection{Spectroid point of view.}

Recall that spectroid is a (small) $ k$-linear hom-finite category  $\mathcal{S}$ such that its different objects are not isomorphic and the algebra of endomorphisms  ${\mathcal{S}}(s,s)$ is local for any  $s\in \mathcal{S}$. For any spectroid $\mathcal{S}$ we consider the algebra  $A(\mathcal{S})=\bigoplus_{s,t\in \mathcal{S}} \mathcal{S}(s,t).$ If $\mathcal{S}$ has finitely many objects, $A(\mathcal{S})$ is a unital algebra  and  $1=\sum_{s\in \mathcal{S}} {\rm id}_s.$

Let  $\mathcal{T}$ be a full subspectroid of $\mathcal{S}$. For any two objects $s,s'\in {\rm Ob}(\mathcal{S})$ we denote by $I_{\mathcal{T}}(s,s')$ the subset  $\mathcal{S}(s,s')$ consisting of morphisms of the form  $\sum f_i,$ where $f_i$ can be presented as a composition $s\to t_i\to s',$ for some $t_i\in \mathcal{T}.$ 
Then denote by $\mathcal{S}/\mathcal{T}$ the category such that  ${\rm Ob}(\mathcal{S}/\mathcal{T})={\rm Ob}(\mathcal{S})\setminus {\rm Ob}(\mathcal{T}),$ and $(\mathcal{S}/\mathcal{T})(s,s')=\mathcal{S}(s,s')/I_{\mathcal{T}}(s,s').$ It is easy to verify  that the composition is well-defined and $\mathcal{S}/\mathcal{T}$ is a spectroid.

For a translation quiver $\Gamma=(\mathcal{Q},\tau, \sigma)$ we denote by  $ k(\Gamma)$ the mesh-spectroid (mesh-category) of $\Gamma$ (see \cite{Bongartz--Gabriel(1982)}).

Let $A$ be a finite dimensional algebra. Denote by ${\rm G}(A)$ the full subcategory of ${\rm mod\text{-}} A,$ whose objects form a complete set of nonisomorphic projective indecomposable modules. It is clear that ${\rm G}(A)$ is a spectroid. If $A$ is a basic algebra, than $A({\rm G}(A))\cong A.$ We denote by  ${\rm ind}(A)$ the full subcategory of ${\rm mod\text{-}} A$, whose objects form a complete set of nonisomorphic indecomposable modules. We will always assume that ${\rm G}(A)\subset {\rm ind}(A).$  We denote by $\underline{\rm ind}(A)$ the full subcategory of ${\rm \underline{mod}\text{-}} A$, whose objects form a complete set of nonisomorphic nonprojective  indecomposable modules. In other words,  $\underline{\rm ind}(A)={\rm ind}(A)/{\rm G}(A).$

Let $A$ be a basic algebra and  $e_1,\dots, e_n$ be a complete set of primitive orthogonal idempotents.  Then  $\{P_i=e_iA\mid  1 \leq i\leq n\}$ is a complete set of nonisomorphic projective indecomposable modules, and we will assume that ${\rm Ob}({\rm G}(A))=\{P_i\mid  1 \leq i\leq n\}.$ It is clear that  ${\rm G}(A)(P_i,P_j)\cong e_jAe_i.$ Let  $\varphi$ be an automorphism such that $\varphi(\{e_1,\dots, e_n\})=\{e_1,\dots, e_n\}.$ For the sake of simplicity,  the corresponding permutation on the set $\{1,\dots, n\}$ we denote by the same symbol $\varphi.$ Then $\varphi(e_i)=e_{\varphi(i)}.$ For a such $\varphi$ we define the autofunctor ${\rm G}(\varphi):{\rm G}(A)\to {\rm G}(A)$  by the formulas  ${\rm G}(\varphi)(P_i)=P_{\varphi^{-1}(i)}$ and ${\rm G}(\varphi)(a\cdot)=\varphi^{-1}(a)\cdot.$ 

An autofunctor $\Phi:{\rm ind}(A)\to {\rm ind}(A)$ is called an {\it extension of} ${\rm G}(\varphi)$ if  $\Phi({\rm G}(A))={\rm G}(A)$ and $\Phi|_{{\rm G}(A)}={\rm G(\varphi)}.$ Such a functor induces an autofunctor on $\underline{\rm ind}(A)={\rm ind}(A)/{\rm G}(A)$ which we denote by  $\underline{\Phi}.$

\begin{lemma}
Let $A$ be a basic algebra,  $e_1,\dots, e_n$ be a complete set of primitive orthogonal idempotents, $\varphi$ be an automorphism of $A$ such that $\varphi(\{e_1,\dots, e_n\})=\{e_1,\dots, e_n\}.$ Then the following conditions hold.
\begin{enumerate}
\item There exists an extension $\Phi: {\rm ind}(A) \to {\rm ind}(A) $ of  ${\rm G}(\varphi)$.
\item If $\Phi$ is an extension of the functor  ${\rm G}(\varphi)$ on ${\rm ind}(A)$, then  $\iota \circ \Phi\cong {\rm res}_{\varphi} \circ \iota.$
$$\xymatrix{
{\rm ind}(A)\ar[rr]^{\iota}\ar[d]^{\Phi} & & {\rm mod\text{-}} A\ar[d]^{{\rm res}_{\varphi}} \\
{\rm ind}(A)\ar[rr]^{\iota} & & {\rm mod\text{-}} A
}$$
\item If $\Phi$ and $\Phi'$ are two extensions of ${\rm G}(\varphi)$ on ${\rm ind}(A)$, there is an isomorphism of functors  $\eta:\Phi\overset{\cong}{\to} \Phi'$ such that \hbox{$\eta_{P_i}={\rm id}_{P_i}.$}
\end{enumerate}
\end{lemma}
\begin{proof} 
(1) The autofunctor ${\rm res}_{\varphi}$ maps indecomposable modules to indecomposable modules. Thus,   for any indecomposable module $M$ there exists a module in ${\rm ind}(A)$ isomorphic to $M_{\varphi}.$ Denote it by $\Phi(M).$ The equality  $e_i=\varphi(e_{\varphi^{-1}(i)})$ includes  $\varphi(P_{\varphi^{-1}(i)})=P_i.$ The restriction of $\varphi$ on $P_{\varphi^{-1}(i)}$ is a module isomorphism, we denote it by $\varphi_{P_i}:P_{\varphi^{-1}(i)}\to (P_i)_{\varphi}.$ Thus, $\Phi(P_i)=P_{\varphi^{-1}(i)}.$ Let us choose isomorphisms $\theta_M:\Phi(M)\to M_\varphi $ such that  $\theta_{P_i}=\varphi_{P_i}.$ Then we can define the autofunctor $\Phi$ on a morphism $f:M\to N$ by the formula $\Phi(f)=\theta_N^{-1}\circ {\rm res}_{\varphi}(f) \circ \theta_M.$ It is a well-defined functor. Moreover, $\Phi(P_i)=P_{\varphi^{-1}(i)}={\rm G}(\varphi)(P_i),$ and for  $a\cdot :P_i\to P_j$ we have
$$\Phi(a\cdot)(b)=(\theta_{P_j}^{-1} \circ{\rm res}_{\varphi}(a\cdot) \circ \theta_{P_i})(b)=(\varphi_{P_j}^{-1}\circ {\rm res}_{\varphi}(a\cdot) \circ \varphi_{P_i})(b)=$$
$$=\varphi^{-1}(a\varphi(b))=\varphi^{-1}(a)b={\rm G}(\varphi)(a\cdot)(b),$$ and,  consequently, $\Phi(a\cdot)={\rm G}(\varphi)(a\cdot).$ Therefore, $\Phi$ is an extension of ${\rm G}(\varphi)$ on ${\rm ind}(A).$ 

(2) Let $\Phi$ be an extension of ${\rm G}(\varphi)$ on ${\rm ind}(A).$ Consider an indecomposable module  $M$ and an element $m\in M.$ Then $me_i\cdot:P_i\to M$ is a module homomorphism. Then the homomorphism  $\Phi(me_i\cdot ):P_{\varphi^{-1}(i)}\to \Phi(M)$ is a multiplication homomorphism by an element which we denote by $\tau_M(me_i)\in \Phi(M)e_{\varphi^{-1}(i)}.$ We define the map $\tau_M:M_\varphi\to \Phi(M)$ by the formula $\tau_M(m)=\sum_i \tau_M(me_i).$ The linear map  $\tau_M$ is a composition of the following isomorphisms
$$M\overset{\cong}{\longrightarrow} {\rm Hom}_A(A,M) \overset{\cong}{\longrightarrow} \bigoplus_i {\rm Hom}_A(P_i,M)\overset{\Phi}{\longrightarrow}$$
$$\overset{\Phi}{\longrightarrow} \bigoplus_i {\rm Hom}_A(P_{\varphi^{-1}(i)},\Phi(M))\overset{\cong}{\longrightarrow} {\rm Hom}_A(A,\Phi(M))\overset{\cong}{\longrightarrow} \Phi(M).$$
Let us check that $\tau_M:M_\varphi\to \Phi(M)$ is a module homomorphism. We have:
$$\tau_M(m\star ae_i)\cdot
=\tau_M(m\varphi(a)e_{\varphi(i)})\cdot =\Phi(m\varphi(a)e_{\varphi(i)}\cdot)=
\sum_k \Phi(me_{k}e_k\varphi(a)e_{\varphi(i)}\cdot)=$$
$$=\sum_k\Phi(me_k\cdot) \Phi(e_k\varphi(a)e_{\varphi(i)}\cdot)=
\sum_k\tau_M(me_k)e_{\varphi^{-1}(k)}ae_i\cdot=\tau_M(m)ae_{i}\cdot.$$
Summing over $i,$ we obtain $\tau_M(m\star a)=\tau_M(m)a.$
Thus, $\tau_M:M_\varphi\to \Phi(M)$ is a module isomorphism. Now we will check that $\tau=\{\tau_M\}_M$ is a morphism of functors ${\rm res}_\varphi \circ\iota\to \iota\circ  \Phi.$ Let $f:M\to N$ be a homomorphism of indecomposable modules in ${\rm ind}(A).$ Let us verify that the following diagram is commutative.
$$\xymatrix{
M_\varphi\ar[rr]^f\ar[d]^{\tau_M} & & N_\varphi\ar[d]^{\tau_N} \\
\Phi(M)\ar[rr]^{\Phi(f)} & & \Phi(N)
}$$
We have:
$$\tau_N(f(m)e_i)\cdot=\Phi(f(m)e_i\cdot)=\Phi(f(me_i))=$$
$$=\Phi(f)  \Phi(me_i\cdot)=\Phi(f) (\tau_M(me_i)\cdot).$$
Applying the last equality to $e_{\varphi^{-1}(i)},$ we obtain 
$$\tau_N(f(m)e_i)=\Phi(f)(\tau_M(me_i)).$$
Summing over $i,$ we get $\tau_N(f(m))=\Phi(f)(\tau_M(m)).$

(3) Consider isomorphisms $\tau_M:M_\varphi \to \Phi(M)$ and $\tau'_M : M_\varphi\to \Phi'(M)$ constructed in the previous part of the proof.   It is easy to check that isomorphisms  $\eta_M=\tau'_M\tau^{-1}_M:\Phi(M)\to \Phi'(M)$ form a natural isomorphism  $\eta=\{\eta_M\}:\Phi\to \Phi'$ satisfying the required property. 
\end{proof}

\begin{proposition}\label{prop_spectroids}
Let $A$ be a basic self-injective algebra,  $e_1,\dots, e_n$ be a complete set of primitive orthogonal idempotents, $\varphi$ be an automorphism of $A$ such that  $\varphi(e_i)=e_i$ for all $i,$ and $\Phi$ be an extension of ${\rm G}(\varphi)$ on ${\rm ind}(A).$  Then $\varphi\in \underline{\rm Inn}(A)$ if and only if $\underline{\Phi}\cong{\rm Id}_{\underline{\rm ind}(A)}.$
\end{proposition}
\begin{proof}
From the last lemma it follows that the following diagram of functors is (weak) commutative.
$$\xymatrix{ 
\underline{\rm ind}(A)\ar[rr]^{\underline{\iota}}\ar[d]^{\underline{\Phi}} & & {\rm \underline{mod}\text{-}} A\ar[d]^{\underline{\rm res}_\varphi} \\
\underline{\rm ind}(A)\ar[rr]^{\underline{\iota}} & & {\rm \underline{mod}\text{-}} A 
}$$ The functor $\underline{\iota}$ is an inclusion. Hence, if $\varphi$ is a stably inner automorphism, then $\underline{\Phi}\cong {\rm Id}.$

On the other hand, if $\underline{\Phi}\cong {\rm Id},$  then $\underline{\rm res}_\varphi \circ \underline{\iota}\cong \underline{\iota}.$ Hence, using the Krull-Schmidt theorem, we obtain $\underline{\rm res}_{\varphi}\cong {\rm Id}.$
\end{proof}

\subsection{Example: the algebra $ k[t]/t^n$.}

In this subsection we compute the group of stably inner automorphisms of the algebra $ k[t]/t^n.$ For a real number $r$ we denote by $\lceil r \rceil$ the least integer greater than or equal to $r.$

\begin{proposition}\label{prop_k[t]/t^n}
Let $\varphi$ be an automorphism of the algebra  $A= k[t]/t^n.$ Then $\varphi$ is stably inner if and only if it induces the identity automorphism on the quotient algebra $ k[t]/t^s,$ where $s={\lceil \frac{n}{2} \rceil}.$ In particular, if  $n \geq 3$ и $\varphi(t)=\sum_{i=1}^{n-1}a_it^i,$ then $\varphi\in\underline{\rm Inn}(A)$ if and only if  $a_1=1$ and  $a_i=0$ for $2\leq i\leq \lceil \frac{n}{2} \rceil-1.$
\end{proposition}

\begin{proof}

If $n=2,$ any automorphism is inner modulo socle, hence, it is stably inner. Therefore, the statement is obvious. From now on we will assume that $n\geq 3.$

The algebra $A= k[t]/t^n$ is a Nakayama algebra, and hence, any indecomposable module has the form  $ k[t]/t^i,$ where $1\leq i\leq n.$ It is well-known that $A$ is a standard algebra. Thus, ${\rm ind}(A)$  is isomorphic to the mesh-spectroid  $ k(\Gamma),$ where $\Gamma$ is the following translation quiver 
$$\Gamma: \hspace{5pt} \xymatrix{
0\ar@<2pt>[r]^{\alpha_0}& 1\ar@<2pt>[l]^{\beta_0}\ar@<2pt>[r]^{\alpha_1} &  \hspace{4pt} \dots \hspace{4pt} \ar@<2pt>[l]^{\beta_1}\ar@<2pt>[r]^{\alpha_{n-3}} & n-2\ar@<2pt>[l]^{\beta_{n-3}}\ar@<2pt>[r]^{\alpha_{n-2}} & n-1\ar@<2pt>[l]^{\beta_{n-2}}}
$$ and the translation $\tau:\Gamma_0\setminus \{0\}\to \Gamma_0\setminus \{0\}$  acts trivially.  A vertex  $i$ of the quiver corresponds to the module  $ k[t]/t^{n-i}.$  Hence,  ${\rm G}(A)$ is a full subspectroid in $ k(\Gamma)$ with the single object $0.$  Moreover,  ${\rm G}(A)(0,0)\cong  k[t]/t^n,$ and this isomorphism maps  $t$ to $\beta_{0} \alpha_{0}.$ We will identify  ${\rm G}(A)(0,0)$ with $ k[t]/t^n$, and  $t$ with $\beta_{0} \alpha_{0}.$  Further, the spectroid  $\underline{\rm ind}(A)={\rm ind}(A)/{\rm G}(A)$ is isomorphic to the mesh-spectroid  $ k(\Gamma^s),$ where  $\Gamma^s$ is the following translation quiver
$$\Gamma^s: \hspace{5pt} \xymatrix{
1\ar@<2pt>[r]^{\alpha_1}& 2\ar@<2pt>[l]^{\beta_1}\ar@<2pt>[r]^{\alpha_2} &  \hspace{4pt} \dots \hspace{4pt} \ar@<2pt>[l]^{\beta_{2}}\ar@<2pt>[r]^{\alpha_{n-3}} & n-2\ar@<2pt>[l]^{\beta_{n-3}}\ar@<2pt>[r]^{\alpha_{n-2}} & n-1\ar@<2pt>[l]^{\beta_{n-2}}}
$$ and the translation $\tau={\rm id}_{\Gamma^s_0}$ acts trivially on all vertices. The syzygy functor $\Omega$ induces an autofunctor $\omega$ on the spectroid $ k(\Gamma^s),$ which acts by the formula  $\omega(i)=n-i,$ $\omega(\alpha_i)=(-1)^{n-i-1}\beta_{n-i-1}$ and $\omega(\beta_i)=(-1)^i\alpha_{n-i-1}.$ 

Let $\varphi$ be an automorphism of the algebra $A.$ It is completely defined by the element $\varphi(t)$. Let  $\varphi(t)=\sum_{k=1}^{n-1} a_kt^k$ and $a_1\ne 0.$   We set $t_i=(-1)^i \beta_i \alpha_i\in  k(\Gamma)(i,i)$ for  $0\leq i\leq n-2$ and $t_{n-1}=0\in k(\Gamma)(n-1,n-1).$ It is clear that  $\alpha_i  t_i=t_{i+1} \alpha_i$ and $\beta_{i}  t_{i+1}= t_{i} \beta_{i}.$ By definition, we put $T_i=\sum_{k=1}^{n-1} a_kt_i^{k-1}\in  k(\Gamma)(i,i).$ Then  $tT_0=\varphi(t).$ Since $a_1\ne 0,$ the elements $T_i$ are invertible. Then we define an autofunctor  $\Phi$ of the category  $ k(\Gamma)$ as identical on objects, and on arrows  by the formulas 
$\Phi(\alpha_i)=\alpha_i T_i,$ $ \Phi(\beta_i)=\beta_i.$ 
In order to prove that the definition is correct we have to prove that the mesh-ideal is invariant under this functor.
$$\Phi(   \alpha_i\beta_i +   \beta_{i+1}\alpha_{i+1})=  \alpha_i  T_i\beta_i+ \beta_{i+1}\alpha_{i+1}  T_{i+1}= (\alpha_i  \beta_i+ \beta_{i+1}\alpha_{i+1})  T_{i+1}$$
$$\Phi(\alpha_{n-2}\beta_{n-2})=\alpha_{n-2}T_{n-2}\beta_{n-2}=
\alpha_{n-2}\beta_{n-2}T_{n-1}$$
The endofunctor $\Phi$ is an autofunctor because $T_i$ is invertible for any $i.$  Since  $\Phi(t)=\Phi(\beta_0\alpha_0)=tT_0=\varphi(t),$ the autofunctor is an extension of  ${\rm G}(\varphi^{-1}).$ The automorphism $\varphi$ is stably inner if and only if $\varphi^{-1}$ is stably inner. Hence, the automorphism $\varphi$ is stably inner if and only if the induced functor  $\underline{\Phi}$ of the spectroid  $ k(\Gamma^s)$  is isomorphic to the identity functor.

We denote by  $\text{\b{$t$}}_i$ and $\text{\b{$T$}}_i$ the images of $t_i$ and $T_i$ in the spectroid $ k(\Gamma^s).$ It is clear that  $\omega(\text{\b{$t$}}_i)=\text{\b{$t$}}_{n-i}$ and $\omega(\text{\b{$T$}}_i)=\text{\b{$T$}}_{n-i}$.  The endomorphism algebras of the spectroid  $ k(\Gamma^s)$ have the form  $ k(\Gamma^s)(i,i)\cong  k[\text{\b{$t$}}_i]/\text{\b{$t$}}_i^{{\rm min}(i, n-i)}.$ Let us consider two cases.

{\it Case 1.} $n=2m.$ Assume that  $\varphi\in \underline{\rm Inn}(A).$ Consider the ideal  $(t^m)\triangleleft A.$ It is clear that $(0:(t^m))+(t^m)=(t^m).$ Then the corollary  \ref{cor_st_in}  gives that  $\varphi$ induces the identity automorphism on $ k[t]/t^m.$ 

Assume that $\varphi$ induces the identity automorphism on  $ k[t]/t^m.$ In other words, $\varphi(t)-t \in (t^m).$ It follows that $T_i-e_i\in (t_i^{m-1}).$  For $i\ne m$ we have ${\rm min}(i,n-i)\leq m-1,$ this implies  $\text{\b{$T$}}_i=e_i$ for $i\ne m.$ Hence, $\underline{\Phi}(\alpha_i)=\alpha_i$ for $i\ne m$ and $\underline{\Phi}(\alpha_m)=\alpha_m \text{\b{$T$}}_m=\text{\b{$T$}}_{m+1}\alpha_m=\alpha_m.$ Therefore, $\underline{\Phi}={\rm Id}.$

{\it Case 2.}  $n=2m+1.$ Since $A$ is a symmetric algebra, $\Omega$ is a Serre functor. Thus, $ k(\Gamma^s)(i,j)\cong D  k(\Gamma^s)(j,\omega(i)),$ and,  consequently, dimensions of four vector spaces  $ k(\Gamma^s)(i,j),$ for $i,j\in\{m,m+1\}$ are equal. It follows from the mesh-relations that composition maps  $\beta_m$
$$(\beta_{m})_*: k(\Gamma^s)(m+1,m+1)\to   k(\Gamma^s)(m+1,m),$$ $$ (\beta_m)^*: k(\Gamma^s)(m,m)\to  k(\Gamma^s)(m+1,m)$$ are surjective, and hence they are isomorphisms. It follows that for any  $y\in  k(\Gamma^s)(m,m)$ there exists a unique  $x\in  k(\Gamma^s)(m+1,m+1)$ such that  $\beta_mx=y\beta_m,$ and, moreover,  $x=\omega(y).$ Similarly, we have that for any  $x\in  k(\Gamma^s)(m+1,m+1)$ there exists unique $y\in  k(\Gamma^s)(m,m)$ such that  $x\alpha_m=\alpha_my,$ and, moreover, $y=\omega(x).$

Let $\varphi$ be a stably inner automorphism. Then there is a functor isomorphism  $\xi: {\rm Id} \overset{\cong}{\longrightarrow} \underline{\Phi}.$ Since $\underline{\Phi}(\beta_m)=\beta_m,$ we obtain  $\beta_m\xi_{m+1}=\xi_m\beta_m.$ Thus we have  $\xi_{m+1}=\omega(\xi_m).$ Using the commutativity of the algebra  $ k(\Gamma^s)(m,m)$, from the equality  $\xi_{m+1}\alpha_m=\underline{\Phi}(\alpha_m)\xi_m=\alpha_m\text{\b{$T$}}_m\xi_m,$ we obtain  $\xi_{m+1}\alpha_m=\alpha_m\xi_m\text{\b{$T$}}_m=
\omega(\xi_m)\alpha_m\text{\b{$T$}}_m=
\xi_{m+1}\alpha_m\text{\b{$T$}}_m.$
Multiplying by $\xi_{m+1}^{-1}$, we get  $\alpha_m=\alpha_m \text{\b{$T$}}_m.$ It follows that $\text{\b{$T$}}_m=\omega(e_{m+1})=e_m$. Hence, $T_m-e_m\in (t_m^m).$ Thus, we obtain  $T_0-e_0\in (t^m).$ Multiplying by $t,$ we get $\varphi(t)-t\in (t^{m+1}).$ It follows that $\varphi$ induces the identity automorphism on  $ k[t]/t^{m+1}.$

Let now $\varphi$ induce the identity automorphism on $ k[t]/t^{m+1}.$  Then from $\varphi(t)-t\in (t^{m+1})$, it follows that $T_i-e_i\in (t_i^m)$. From the inequality  ${\rm min}(i,n-i)\leq m ,$ we obtain  $\text{\b{$T$}}_i=e_i.$ Hence, $\underline{\Phi}={\rm Id},$ and $\varphi\in \underline{\rm Inn}(A).$
\end{proof}

\section{Self-injective algebras of finite  type.}

This section is devoted to the calculation of stable Calabi-Yau dimensions for the  rest three cases of self-injective algebras of finite representation type over algebraically closed field $k$.  Recall that $m_{\Delta}$ is equal to $n$, $2n-3$, 11, 17 or 29 when $\Delta$ is $A_n$, $D_n$, $E_6$, $E_7$ or $E_8$ respectively.

Let $\mathcal{Q}=(\mathcal{Q}_0,\mathcal{Q}_1,s,t)$ be a quiver and $A=k\mathcal{Q}/I$ be a bound quiver algebra (where $I$ is an admissible ideal). For $u,v\in \mathcal{Q}_0$ we denote by $P_{[v][u]}:= A e_{v}\otimes e_{u}A$ the indecomposable projective bimodule corresponding to the idempotent \hbox{$e_{v}\otimes e_{u}$}. We set $\mathcal{P}_0(A^e)$ to be the category of finitely generated projective bimodules of the form $\bigoplus P_{[v_i][u_i]}$. Let  $\sigma$ be an automorphism of the algebra $A$ such that $\sigma(\{e_v\mid v\in\mathcal{Q}_0\})=\{e_v\mid v\in\mathcal{Q}_0\}.$ We define the  functor $\sigma:\mathcal{P}_0(A^e)\rightarrow\mathcal{P}_0(A^e)$ by the following formulas: $\sigma\left(\bigoplus P_{[v_i][u_i]}\right)=\bigoplus P_{[\sigma(v_i)][u_i]}$ (here $\sigma(e_{v})=e_{\sigma(v)}$), 
and if $d:\bigoplus P_{[v_i][u_i]}\longrightarrow \bigoplus P_{[v'_j][u'_j]}$ is given by 
$d(e_{v_i}\otimes e_{u_i})= \sum x_{l,j}\otimes y_{l,j},$ then $\sigma(d)(e_{\sigma(v_i)}\otimes e_{u_i})=\sum \sigma(x_{l,j})\otimes y_{l,j}$.

\subsection{Algebras of type $(A_{2n+1}, r, 2)$}

Throughout this subsection, we assume that $n$ is even. Any algebra of type $(A_{2n+1}, r, 2)$ is derived equivalent to the algebra $A=k\mathcal Q/I$ where the bound quiver $(\mathcal Q,I)$ is constructed as follows. The set of vertices is $\mathcal Q_0=(\mathbb{Z}/2r\mathbb{Z}\times\{1\dots n\})\sqcup\,\mathbb{Z}/r\mathbb{Z}$. For $i\in\mathbb{Z}$ the corresponding element of  $\mathbb{Z}/r\mathbb{Z}$ is denoted again by $i$ and the corresponding element of $\mathbb{Z}/2r\mathbb{Z}$ is denoted by $\,\,\widehat{i}$. The set of arrows $\mathcal Q_1$ of the quiver $\mathcal Q$ consists of the following elements:
$$
\begin{aligned}
&\alpha_{\,\,\widehat{i},j}:(\,\,\widehat{i},j)\rightarrow (\,\,\widehat{i},j+1)\hspace{0.2cm}(1\leq i\leq 2r, 1\leq j\leq n-1),\\
&\alpha_{\,\,\widehat{i},0}:i\rightarrow (\,\,\widehat{i},1), \alpha_{\,\,\widehat{i},n}:(\,\,\widehat{i},n)\rightarrow i+1\hspace{0.2cm}(1\leq i\leq 2r).
\end{aligned}
$$
The ideal $I$ is generated by the elements
$$
 \begin{aligned}
&\alpha_{\,\,\widehat{i+r+1},0}\alpha_{\,\,\widehat{i},n}, \hspace{0.2cm}(1\leq i\leq 2r),\alpha_{\,\,\widehat{i},n}\dots\alpha_{\,\,\widehat{i},0}+\alpha_{\,\,\widehat{i+r},n}\dots\alpha_{\,\,\widehat{i+r},0}\hspace{0.2cm}(1\leq i\leq 2r),\\
 &\alpha_{\,\,\widehat{i+1},j}\dots\alpha_{\,\,\widehat{i+1},0}\alpha_{\,\,\widehat{i},n}\dots\alpha_{\,\,\widehat{i},j}\hspace{0.2cm}(1\leq i\leq 2r, 1\leq j\leq n-1).
 \end{aligned}
 $$

$$
  
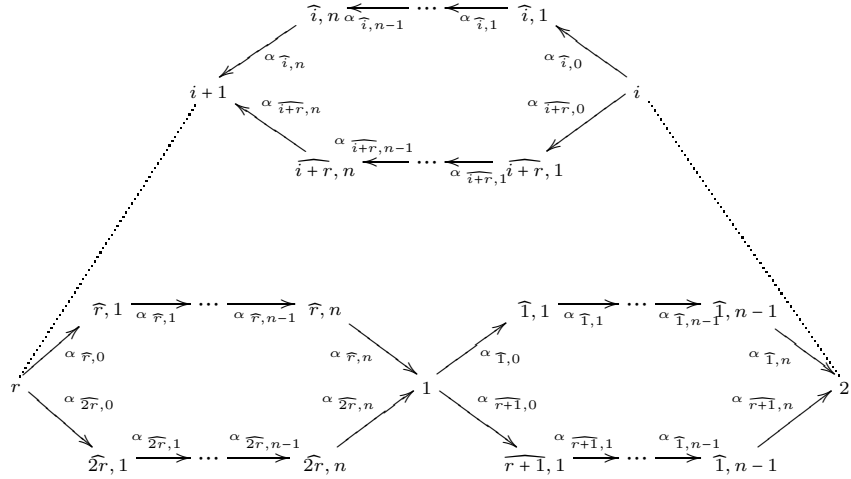
\xymatrix@=0.6cm
  {
  &&&\txt{\scriptsize $\,\,\widehat{i},n$}\ar@{->}[dl]^*\txt{\tiny $\alpha_{\,\,\widehat{i},n}$}&\cdots\ar@{->}[l]^*\txt{\tiny $\alpha_{\,\,\widehat{i},n-1}$}&\txt{\scriptsize $\,\,\widehat{i},1$}\ar@{->}[l]^*\txt{\tiny $\alpha_{\,\,\widehat{i},1}$}\\
  &&\txt{\scriptsize $i+1$}&&&&
  \txt{\scriptsize $i$}\ar@{->}[ul]^*\txt{\tiny $\alpha_{\,\,\widehat{i},0}$}\ar@{->}[dl]_*\txt{\tiny $\alpha_{\,\,\widehat{i+r},0}$}\\
  &&&\txt{\scriptsize $\,\,\widehat{i+r},n$}\ar@{->}[ul]_*\txt{\tiny $\alpha_{\,\,\widehat{i+r},n}$}&\cdots\ar@{->}[l]_*\txt{\tiny $\alpha_{\,\,\widehat{i+r},n-1}$}&\txt{\scriptsize $\,\,\widehat{i+r},1$}\ar@{->}[l]^*\txt{\tiny $\alpha_{\,\,\widehat{i+r},1}$}\\
  \\
  &\txt{\scriptsize $\,\,\widehat{r},1$}\ar@{->}[r]_*\txt{\tiny $\alpha_{\,\,\widehat{r},1}$}&\cdots\ar@{->}[r]_*\txt{\tiny $\alpha_{\,\,\widehat{r},n-1}$}&\txt{\scriptsize $\,\,\widehat{r},n$}\ar@{->}[dr]_*\txt{\tiny $\alpha_{\,\,\widehat{r},n}$}&&\txt{\scriptsize $\,\,\widehat{1},1$}\ar@{->}[r]_*\txt{\tiny $\alpha_{\,\,\widehat{1},1}$}&\cdots\ar@{->}[r]_*\txt{\tiny $\alpha_{\,\,\widehat{1},n-1}$}&\txt{\scriptsize $\,\,\widehat{1},n-1$}\ar@{->}[dr]_*\txt{\tiny $\alpha_{\,\,\widehat{1},n}$}&\\
  \txt{\scriptsize $r$}\ar@{.}[uuuurr]\ar@{->}[ur]_*\txt{\tiny $\alpha_{\,\,\widehat{r},0}$}\ar@{->}[dr]^*\txt{\tiny $\alpha_{\,\,\widehat{2r},0}$}&&&&\txt{\scriptsize 1}\ar@{->}[ur]_*\txt{\tiny $\alpha_{\,\,\widehat{1},0}$}\ar@{->}[dr]^*\txt{\tiny $\alpha_{\,\,\widehat{r+1},0}$}&&&&\txt{\scriptsize 2}\ar@{.}[uuuull]\\
  &\txt{\scriptsize $\,\,\widehat{2r},1$}\ar@{->}[r]^*\txt{\tiny $\alpha_{\,\,\widehat{2r},1}$}&\cdots\ar@{->}[r]^*\txt{\tiny $\alpha_{\,\,\widehat{2r},n-1}$}&\txt{\scriptsize $\,\,\widehat{2r},n$}\ar@{->}[ur]^*\txt{\tiny $\alpha_{\,\,\widehat{2r},n}$}&&\txt{\scriptsize $\,\,\widehat{r+1},1$}\ar@{->}[r]^*\txt{\tiny $\alpha_{\,\,\widehat{r+1},1}$}&\cdots\ar@{->}[r]^*\txt{\tiny $\alpha_{\,\,\widehat{1},n-1}$}&\txt{\scriptsize $\,\,\widehat{1},n-1$}\ar@{->}[ur]^*\txt{\tiny $\alpha_{\,\,\widehat{r+1},n}$}
  }
$$

For the sake of simplicity, we omit brackets in the notation of vertices.
Let us define an automorphism $\sigma$ of the algebra $A$ on generators:
$$
\begin{array}{c}
\sigma(e_i)=e_{i+n+1}\hspace{0.2cm}(1\leq i\leq r),\ \sigma(e_{\,\,\widehat{i},j})=e_{\,\,\widehat{i+r+n+1},j}(1\leq i\leq 2r,1\leq j\leq n),\\
\sigma(\alpha_{\,\,\widehat{i},j})=\alpha_{\,\,\widehat{i+r+n+1},j}(1\leq i\leq 2r,1\leq j\leq n-1),\\
\sigma(\alpha_{\,\,\widehat{i},0})=\alpha_{\,\,\widehat{i+r+n+1},0}(r\leq i\leq 2r-1),\ \sigma(\alpha_{\,\,\widehat{i},0})=-\alpha_{\,\,\widehat{i+r+n+1},0}(0\leq i\leq r-1),\\
\sigma(\alpha_{\,\,\widehat{i},n})=\alpha_{\,\,\widehat{i+r+n+1},n}(r-1\leq i\leq 2r-2),\\
\sigma(\alpha_{\,\,\widehat{i},n})=-\alpha_{\,\,\widehat{i+r+n+1},n}(-1\leq i\leq r-2).
\end{array}
$$

Recall the description of the terms $Q_t$ ($t\geq 0$) of the minimal bimodule resolution of the algebra $A$ given in \cite{Gen_Kach}. The terms $Q_t$ ($0\leq t\leq 2n$) are described by the following formulas:
 $$
 \begin{aligned}
Q_{2m}&= \bigoplus\limits_{i=1}^rP_{[i+m][i]}
\oplus \bigoplus\limits_{i=1}^{2r}
\bigg(
   \Big(
        \bigoplus\limits_{j=1}^{n-m}P_{[\,\,\widehat{i+m},j+m][\,\,\widehat{i},j]}
  \Big)
    \oplus
  \Big(
        \bigoplus\limits_{j=n-m+1}^{n}P_{[\,\,\widehat{i+r+m},j+m-n][\,\,\widehat{i},j]}
  \Big)
  \bigg)\\
  &(m=0,\dots,n),\\
 Q_{2m+1}&=\bigoplus\limits_{i=1}^{2r}
\bigg(
P_{[\,\,\widehat{i+m},m+1][i]}\oplus
   \Big(
        \bigoplus\limits_{j=1}^{n-m-1}P_{[\,\,\widehat{i+m},j+m+1][\,\,\widehat{i},j]}
  \Big)
    \oplus
    P_{[i+m+1][\,\,\widehat{i},n-m]}\\
    &\oplus
  \Big(
        \bigoplus\limits_{j=n-m+1}^{n}P_{[\,\,\widehat{i+r+m+1},j+m-n][\,\,\widehat{i},j]}
  \Big)
  \bigg)\hspace{0.5cm}(m=0,\dots,n-1).
\end{aligned}
$$
Moreover, the terms with numbers exceeding $2n$ are obtained by the formula $Q_{l(2n+1)+t}=\sigma^l(Q_t)$ ($0\leq t\leq 2n$, $l\geq 1$).
Moreover, the fact that $(2n+1)$-th syzygy $\Omega^{2n+1}_{A^e}(A)$ of the bimodule $A$ is isomorphic to the twisted module 
 ${}_{\sigma^{-1}}A\cong A_{\sigma}$ is noted in the same paper (see also \cite{Erdmann_Holm_Snashall}).

It follows from the corollary \ref{cor_scydim} that in order to compute the stable Calabi-Yau dimension we have to find the least $t\geq 0$ such that $\Omega^{t+1}_{ A^e}(A)\cong A_{\nu^{-1}\varphi}$ where $\nu$ is Nakayama automorphism of the algebra $A$ and $\varphi$ is a stably inner automorphism. In this case, we have $$Q_{t+1}=(Q_0)_{\nu^{-1}\varphi}\cong\bigoplus\limits_{i=1}^rP_{[i+1][i]}\oplus \bigoplus\limits_{i=1}^{2r}
        \bigoplus\limits_{j=1}^{n}P_{[\,\,\widehat{i+1},j][\,\,\widehat{i},j]}$$
  and
$$Q_{t+2}=(Q_1)_{\nu^{-1}\varphi}\cong\bigoplus\limits_{i=1}^{2r}\bigg(P_{[\,\,\widehat{i+1},1][i]}\oplus\Big(
        \bigoplus\limits_{j=1}^{n-1}P_{[\,\,\widehat{i+1},j+1][\,\,\widehat{i},j]}\Big)\oplus P_{[i+2][\,\,\widehat{i},n]}\bigg).$$ 
Then, using the description of $Q_t$ ($t\geq 0$), it is easy to verify that $t+1=l(2n+1)$ for some $l>0$, i.e. $\Omega^{t+1}_{ A^e}(A)\cong A_{\sigma^l}$. It remains to find such numbers $l$ that the automorphism $\nu\sigma^l$ is stably inner. The equality $\nu\sigma^l(e_{\,\,\widehat{i},1})=e_{\,\,\widehat{i+l(r+n+1)-1},1}$ can be easily verified. Consequently, if the automorphism $\nu\sigma^l$ is stably inner then $2r\mid l(r+n+1)-1$. Existence of such $l$ is equivalent to the fact that ${\rm GCD}(r+n+1,2r)=1$.

Suppose that ${\rm GCD}(r+n+1,2r)=1$. Let $2r\mid l(r+n+1)-1$ (in particular $l$ is odd). Let us prove that the automorphism $\nu\sigma^l$ is stably inner. We assume that Nakayama automorphism $\nu$ is constructed using Frobenius form $\varepsilon :A\to k$ which is equal to $-1$ on paths $\alpha_{\,\,\widehat{i},n}\dots\alpha_{\,\,\widehat{i},0}$ ($-1\leq i\leq r-2$), $1$ on all other paths of length $n+1$ and $0$ on all remaining paths. In this case $\nu$ satisfies the following equalities:
$$
\begin{array}{c}
\nu(e_i)=e_{i-1}\hspace{0.2cm}(1\leq i\leq r),\ \nu(e_{\,\,\widehat{i},j})=e_{\,\,\widehat{i-1},j}(1\leq i\leq 2r,1\leq j\leq n),\\
\nu(\alpha_{\,\,\widehat{i},j})=\alpha_{\,\,\widehat{i-1},j}(1\leq i\leq 2r,1\leq j\leq n-1),\\
\nu(\alpha_{\,\,\widehat{i},0})=\alpha_{\,\,\widehat{i-1},0}(r\leq i\leq 2r-1),\ \nu(\alpha_{\,\,\widehat{i},0})=-\alpha_{\,\,\widehat{i-1},0}(0\leq i\leq r-1),\\
\nu(\alpha_{\,\,\widehat{i},n})=\alpha_{\,\,\widehat{i-1},n}(r-1\leq i\leq 2r-2),\ \nu(\alpha_{\,\,\widehat{i},n})=-\alpha_{\,\,\widehat{i-1},n}(-1\leq i\leq r-2).
\end{array}
$$
It is easy to verify the following equalities
$$
\begin{array}{c}
\nu\sigma^l(e_i)=e_{i}\hspace{0.2cm}(1\leq i\leq r),\ \nu\sigma^l(e_{\,\,\widehat{i},j})=e_{\,\,\widehat{i},j}(1\leq i\leq 2r,1\leq j\leq n),\\
\nu\sigma^l(\alpha_{\,\,\widehat{i},j})=\alpha_{\,\,\widehat{i},j}(1\leq i\leq 2r,1\leq j\leq n-1),\\
\nu\sigma^l(\alpha_{\,\,\widehat{i},0})=(-1)^{\sum\limits_{q=i-r+1}^ip_l(q)}\alpha_{\,\,\widehat{i},0}(1\leq i\leq 2r),\ \nu\sigma^l(\alpha_{\,\,\widehat{i},n})=(-1)^{\sum\limits_{q=i-r+2}^{i+1}p_l(q)}\alpha_{\,\,\widehat{i},n}(1\leq i\leq 2r),
\end{array}
$$
where $p_l:\mathbb{Z}/2r\mathbb{Z}\rightarrow \mathbb{Z}$ is defined by the formula
$$
  p_l(q)={\rm card}(\{s : 0\leq s\leq l, 2r\mid q+s(r+n+1)\}).
$$
Since $l$ is odd, it is easy to verify that $\nu\sigma^l(x)=a^{-1}xa$ for all $x\in A$ where
$$
a=\sum\limits_{i=1}^r(-1)^{\sum\limits_{q=i-r+1}^ip_l(q)}e_i+\sum\limits_{i=1}^{2r}\sum\limits_{j=1}^ne_{\,\,\widehat{i},j},
$$
Hence, the automorphism $\nu\sigma^l(x)$ is inner, and hence, it is stably inner. The following proposition follows from the above argument.

\begin{proposition}
Let $A$ be a self-injective algebra of finite representation type of type $(A_{2n+1}, r, 2)$ where $n$ is even, $n>0$. Then the stable Calabi-Yau dimension of the algebra $A$ is finite if and only if ${\rm GCD}(r+n+1,2r)=1$. If the last mentioned condition is satisfied then the stable Calabi-Yau dimension is equal to $l(2n+1)-1$ where $2r\mid l(r+n+1)-1$ and $0<l<2r$.
 \end{proposition}

\subsection{Algebras of type $(D_n, r, 2)$}

Throughout this subsection, we assume that $r$ is even. Any algebra of type $(D_n, r, 2)$ is derived equivalent to the algebra $A=k\mathcal Q/I$ where the bound quiver $(\mathcal Q,I)$ is the following. The set of vertices is $\mathcal Q_0=(\mathbb{Z}/r\mathbb{Z}\times\{1\dots n-2\})\sqcup\,\mathbb{Z}/2r\mathbb{Z}$. We use in this subsection the following notation: for $i\in\mathbb{Z}$ the corresponding element in  $\mathbb{Z}/r\mathbb{Z}$ (resp. in $\mathbb{Z}/2r\mathbb{Z}$) is denoted by $i$ (resp. $\,\,\widehat{i}$). The set of arrows $\mathcal Q_1$ of the quiver $\mathcal Q$ consists of the following elements:
$$
\begin{aligned}
&\gamma_{\,\,\widehat{i}}:(i, n-2)\rightarrow \,\,\widehat{i}, \beta_{\,\,\widehat{i}}:\,\,\widehat{i}\rightarrow (i+1, 1)\hspace{0.2cm}(1\leq i\leq 2r),\\
&\alpha_{i,j}:(i,j)\rightarrow (i,j+1)\hspace{0.2cm}(1\leq i\leq r, 1\leq j\leq n-3).
\end{aligned}
$$
The ideal $I$ is generated by the elements
$$
 \begin{aligned}
&\gamma_{\,\,\widehat{i+r}}\alpha_{i,n-3}\dots\alpha_{i,1}\beta_{\,\,\widehat{i-1}}, \beta_{\,\,\widehat{i}}\gamma_{\,\,\widehat{i}}-\beta_{\,\,\widehat{i+r}}\gamma_{\,\,\widehat{i+r}}\hspace{0.2cm}(1\leq i\leq 2r),\\
 &\alpha_{i+1,j}\dots\alpha_{i+1,1}\beta_{\,\,\widehat{i}}\gamma_{\,\,\widehat{i}}\alpha_{i,n-3}\dots\alpha_{i,j}\hspace{0.2cm}(1\leq i\leq r, 1\leq j\leq n-3).
 \end{aligned}
 $$

$$
  
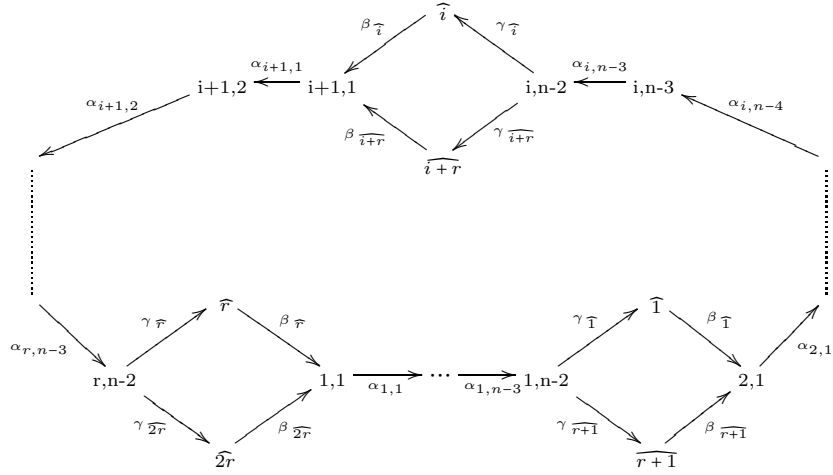
\xymatrix@=0.6cm
  {
  &&&&\txt{\scriptsize $\,\,\widehat{i}$}\ar@{->}[dl]_*\txt{\tiny $\beta_{\,\,\widehat{i}}$}\\
  &&\txt{\scriptsize i+1,2} \ar@{->}[dll]_*\txt{\tiny $\alpha_{i+1,2}$}&\txt{\scriptsize i+1,1}\ar@{->}[l]_*\txt{\tiny $\alpha_{i+1,1}$}&&
  \txt{\scriptsize i,n-2}\ar@{->}[ul]_*\txt{\tiny $\gamma_{\,\,\widehat{i}}$}\ar@{->}[dl]^*\txt{\tiny $\gamma_{\,\,\widehat{i+r}}$}&
  \txt{\scriptsize i,n-3}\ar@{->}[l]_*\txt{\tiny $\alpha_{i,n-3}$}&\\
  &&&&\txt{\scriptsize $\,\,\widehat{i+r}$}\ar@{->}[ul]^*\txt{\tiny $\beta_{\,\,\widehat{i+r}}$}&&&&\ar@{->}[ull]_*\txt{\tiny $\alpha_{i,n-4}$}\\
  \\
  \ar@{.}[uu]\ar@{->}[dr]_*\txt{\tiny $\alpha_{r,n-3}$}&&\txt{\scriptsize $\,\,\widehat{r}$}\ar@{->}[dr]^*\txt{\tiny $\beta_{\,\,\widehat{r}}$}&&&&
  \txt{\scriptsize $\,\,\widehat{1}$}\ar@{->}[dr]^*\txt{\tiny $\beta_{\,\,\widehat{1}}$}&&\ar@{.}[uu]\\
  &\txt{\scriptsize r,n-2}\ar@{->}[ur]^*\txt{\tiny $\gamma_{\,\,\widehat{r}}$}\ar@{->}[dr]_*\txt{\tiny $\gamma_{\,\,\widehat{2r}}$}&&
  \txt{\scriptsize 1,1}\ar@{->}[r]_*\txt{\tiny $\alpha_{1,1}$}&\cdots\ar@{->}[r]_*\txt{\tiny $\alpha_{1,n-3}$}&
  \txt{\scriptsize 1,n-2}\ar@{->}[ur]^*\txt{\tiny $\gamma_{\,\,\widehat{1}}$}\ar@{->}[dr]_*\txt{\tiny $\gamma_{\,\,\widehat{r+1}}$}&&
  \txt{\scriptsize 2,1}\ar@{->}[ur]_*\txt{\tiny $\alpha_{2,1}$}\\
  &&\txt{\scriptsize $\,\,\widehat{2r}$}\ar@{->}[ur]_*\txt{\tiny $\beta_{\,\,\widehat{2r}}$}&&&&\txt{\scriptsize $\,\,\widehat{r+1}$}\ar@{->}[ur]_*\txt{\tiny $\beta_{\,\,\widehat{r+1}}$}
  }
$$

Let us define an automorphism $\sigma$ of the algebra $A$ on generators:
$$
\begin{array}{c}
\sigma(e_{i,j})=e_{i+n-1,j},\ \sigma(\alpha_{i,j})=\alpha_{i+n-1,j}\hspace{0.5cm}(1\leq i\leq r,1\leq j\leq n-2),\\
\sigma(e_{\,\,\widehat{i}})=e_{\,\,\widehat{i+n-1+rn}},\sigma(\beta_{\,\,\widehat{i}})=\beta_{\,\,\widehat{i+n-1+rn}}\hspace{0.5cm}(1\leq i\leq 2r),\\ \sigma(\gamma_{\,\,\widehat{i}})=-\gamma_{\,\,\widehat{i+n-1+rn}}\hspace{0.5cm}(1\leq i\leq 2r, r\nmid i),
\sigma(\gamma_{\,\,\widehat{i}})=\gamma_{\,\,\widehat{i+n-1+rn}}\hspace{0.5cm}(i\in\{r,2r\}).
\end{array}
$$

Recall the description of the terms $Q_t$ ($t\geq 0$) of the bimodule resolution of the algebra $A$ (see \cite{Volk5}). The terms $Q_t$ ($0\leq t\leq 2n-4$) are described by the following formulas:
 $$
 \begin{aligned}
Q_{2m}&=  \bigoplus\limits_{i=1}^r
 \bigg(
    \Big(
        {\:\bigoplus\limits_{j=1}^{n-2-m}P_{[i+m,j+m][i,j]}}
    \Big)
  \oplus
   \Big(
  {\bigoplus\limits_{j=n-1-m}^{n-2}P_{[i+m,j+m-(n-2)][i,j]}
  }
  \Big)\bigg)\\
 &\oplus\bigoplus\limits_{i=1}^{2r}
  P_{[\,\,\widehat{i+m(r+1)}][\,\,\widehat{i}\,\,]}\hspace{0.5cm}
(m=0,\dots,n-2),\\
 Q_{2m+1}&=\bigoplus\limits_{i=1}^r
\bigg(\Big(\bigoplus\limits_{j=1}^{n-3-m}P_{[i+m,j+m+1][i,j]}
\Big)\oplus
\Big(\bigoplus\limits_{j=n-1-m}^{n-2}P_{[i+m+1,j+m-(n-2)][i,j]}
 \Big)\bigg)\\
 &\oplus\bigoplus\limits_{i=1}^{2r}\Big(P_{[\,\,\widehat{i+m}][i,n-2-m]}\oplus P_{[i+m+1,m+1][\,\,\widehat{i}\,\,]}\Big)\hspace{0.5cm}(m=0,\dots,n-3).
\end{aligned}
$$
Moreover, the terms with numbers exceeding $2n-4$ are obtained by the formula $Q_{l(2n-3)+t}=\sigma^l(Q_t)$ ($0\leq t\leq 2n-4$, $l\geq 1$).
Moreover, the fact that $(2n-3)$-th syzygy $\Omega^{2n-3}_{A^e}(A)$ of the $ A^e$-module $A$ is isomorphic to twisted module  ${}_{\sigma^{-1}}A\cong A_{\sigma}$ is noted in the same paper.

By corollary \ref{cor_scydim} we have to find the least $t\geq 0$ such that $\Omega^{t+1}_{ A^e}(A)\cong A_{\nu^{-1}\varphi}$ where $\nu$ is Nakayama automorphism of the algebra $A$ and $\varphi$ is a stably inner automorphism. In this case we have
$$Q_{t+1}=(Q_0)_{\nu^{-1}\varphi}\cong\bigoplus\limits_{i=1}^r
        {\:\bigoplus\limits_{j=1}^{n-2}P_{[i+1,j][i,j]}}
 \oplus\bigoplus\limits_{i=1}^{2r}P_{[\,\,\widehat{i+1)}][\,\,\widehat{i}\,\,]}$$
  and
$$Q_{t+2}=(Q_1)_{\nu^{-1}\varphi}\cong\bigoplus\limits_{i=1}^r\bigoplus\limits_{j=1}^{n-3}P_{[i+1,j+1][i,j]}
\oplus\bigoplus\limits_{i=1}^{2r}\Big(P_{[\,\,\widehat{i+1}][i,n-2]}\oplus P_{[i+2,1][\,\,\widehat{i}\,\,]}\Big).$$
 Then, using the description of $Q_t$ ($t\geq 0$), it is easy to prove  that $t+1=l(2n-3)$ for some $l>0$, i.e. $\Omega^{t+1}_{ A^e}(A)\cong A_{\sigma^l}$. It remains to find such numbers $l$ that the automorphism $\nu\sigma^l$ is stably inner. The equality $\nu\sigma^l(e_{\,\,\widehat{i}})=e_{\,\,\widehat{i+l(n-1+rn)-1}}$ can be easily verified. Consequently, if the automorphism $\nu\sigma^l$ is stably inner then $2r\mid l(n-1+rn)-1$. Existence of such $l$ is equivalent to the fact that ${\rm GCD}(n-1,r)=1$.

Suppose that ${\rm GCD}(n-1,r)=1$ (in particular $ n$ is even). Let $2r\mid l(n-1+rn)-1$. Then $l$ is odd. In this case the following equalities hold (we assume that Nakayama automorphism $\nu$ is constructed using Frobenius form which equals 1 on paths of length $n-1$ and 0 on all remaining paths of the quiver $\mathcal Q$)
$$
\begin{array}{c}
\nu\sigma^l(e_x)=e_x\hspace{0.5cm}(x\in\mathcal Q_0),\ \nu\sigma^l(\alpha_{i,j})=\alpha_{i,j}\hspace{0.5cm}(1\leq i\leq r,1\leq j\leq n-2),\\
\nu\sigma^l(\beta_{\,\,\widehat{i}})=\beta_{\,\,\widehat{i}}\hspace{0.5cm}(1\leq i\leq 2r),\ \nu\sigma^l(\gamma_{\,\,\widehat{i}})=(-1)^{l+p_l(i)}\gamma_{\,\,\widehat{i}}\hspace{0.5cm}(1\leq i\leq 2r),
\end{array}
$$
where $p_l:\mathbb{Z}/r\mathbb{Z}\rightarrow \mathbb{Z}$ is defined by the formula:
$$
  p_l(i):={\rm card}(\{s\mid  0\leq s\leq l-1, r\mid i+s(n-1)\}).
$$

If ${\rm char} (k)=2$ then $\nu\sigma^l$ is identity automorphism. Suppose that ${\rm char} (k)\ne 2$. It follows from the corollary  \ref{cor_st_inn_loops} that if $\nu\sigma^l$ is stably inner then there are $d_{i,j}\in k^*$ ($1\leq i\leq r$, $1\leq j\leq n-2$) and $d_{\,\,\widehat{i}}\in k^*$ ($1\leq i\leq 2r$) such that
 $$
 d_{i,j}=d_{i,j+1}\hspace{0.2cm}(1\leq i \leq r, 1\leq j\leq n-3),$$ $$\ d_{i+1,1}=d_{\,\,\widehat{i}},\ d_{i,n-2}=(-1)^{l+p_l(i)}d_{\,\,\widehat{i}}\hspace{0.2cm}(1\leq i \leq 2r).
 $$
By induction, it can be easily proved  that $d_{m,1}=(-1)^{(m-1)l+\sum\limits_{q=1}^{m-1}p_l(q)}d_{1,1}$. In particular, we have $d_{1,1}=d_{r+1,1}=(-1)^{rl+\sum\limits_{q=1}^{r}p_l(q)}d_{1,1}=(-1)^ld_{1,1}=-d_{1,1}$, i.e. $d_{1,1}=0\not\in k^*$. The contradiction proves that $\nu\sigma^l$ is not stably inner in the case of ${\rm char} (k) \ne 2$. Thus, we have

\begin{proposition}
Let $A$ be a self-injective algebra of finite representation type of type $(D_n, r, 2)$ where $r$ is even. Then stable Calabi-Yau dimension of the algebra $A$ is finite if and only if ${\rm GCD}(r,n-1)=1$ and ${\rm char}(k)=2$. If this conditions are satisfied then the stable Calabi-Yau dimension is equal to $l(2n-3)-1$ where $2r\mid l(n-1)-1$ and $0<l<2r$.
 \end{proposition}

\subsection{Nonstandard algebras}

It follows from \cite{Asashiba} that all nonstandard algebras are of type $(D_{3n}, \frac{1}{3}, 1)$ and this type determines them up to the derived equivalence. Moreover, nonstandard algebras exist only over fields of characteristic $2$.

Any nonstandard algebra of type $(D_{3n}, \frac{1}{3}, 1)$ is derived equivalent to the algebra $A=k\mathcal Q/I$ where the bound quiver $(\mathcal Q,I)$ is constructed as follows. The set of vertices is $\mathcal Q_0= \{0,\dots, n-1\}$. The set of arrows $\mathcal Q_1$ of the quiver $\mathcal Q$ consists of the following elements:
$$
\begin{aligned}
\beta:0\rightarrow 0,\alpha_i:i\rightarrow i+1(0\leq i\leq n-2), \alpha_{n-1}:n-1\rightarrow 0.
\end{aligned}
$$
Also we introduce the auxiliary notation: $\nu_i=\alpha_{n-1}\dots\alpha_i$, $\mu_i=\alpha_{i-1}\dots\alpha_0$ for $0\leq i\leq n-1$.

The ideal $I$ is generated by the elements $\alpha_0\alpha_{n-1}+\alpha_0\beta\alpha_{n-1}$, $\beta^2-\nu_0$ and $\mu_i\nu_i$,
 where $i=1,\dots, {\rm max}(n-2,1)$.
 $$
  
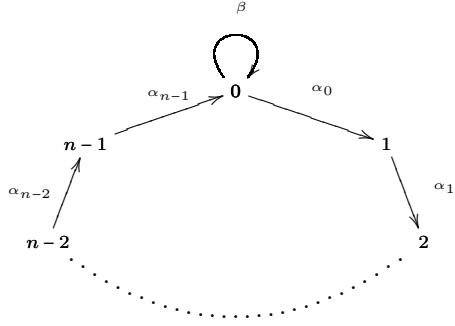
\xy
  (0,20)*+\txt{\scriptsize $n-2$ };
  (5,33)*+\txt{\scriptsize $n-1$}
  **\crv{}
  ?>*\dir{>}
  ?(.3) *!LD!/^-20pt/\txt{\tiny $\alpha_{n-2}$};
  (5,33)*+\txt{\scriptsize $n-1$};
  (25,40)*+\txt{\scriptsize 0}
  **\crv{}
  ?>*\dir{>}
  ?(.5) *!LD!/^-8pt/\txt{\tiny $\alpha_{n-1}$};
  (25,40)*+\txt{\scriptsize $0$};
  (45,33)*+\txt{\scriptsize $1$}
  **\crv{}
  ?>*\dir{>}
  ?(.5) *!LD!/^-8pt/\txt{\tiny $\alpha_0$};
  (25,40)*+\txt{\scriptsize $0$};
  (25,40)*+\txt{\scriptsize $0$}
  **\crv{(20, 45)&(25,50)&(30, 45)}
  ?>*\dir{>}
  ?(.5) *!LD!/^-8pt/\txt{\tiny $\beta$};
  (45,33)*+\txt{\scriptsize $1$};
  (50,20)*+\txt{\scriptsize $2$}
  **\crv{}
  ?>*\dir{>}
  ?(.7) *!LD!/^-10pt/\txt{\tiny $\alpha_1$};
  (50,20)*+\txt{\scriptsize $2$}; (0,20)*+\txt{\scriptsize $n-2$}
  **\crv{~*=<4pt>{.} (25, 0)}
  \endxy
$$

Let us define automorphism $\sigma$ of the algebra $A$ on generators:
$$\sigma(e_i)=e_i,\ \sigma(\alpha_i)=\alpha_i\ (0\leq i\leq n-1),\ \sigma(\beta)=\beta+\beta^2+\beta^3.$$

By definition, we put:
 \begin{align*}
T_{2m}&=
 P_{[0][0]}\oplus
    \Big(
        {\:\bigoplus\limits_{i=1}^{n-1-m}P_{[i+m][i]}}
    \Big)
  \oplus
   \Big(
  {\bigoplus\limits_{i=n-m}^{n-1}P_{[i+m-(n-1)][i]}
  }
  \Big)
  \ \ \ (m=0,\dots,n-1),\\
 T'_{2m+1}&=\Big(\bigoplus\limits_{i=0}^{n-2-m}P_{[i+m+1][i]}
\Big)\oplus
\Big(\bigoplus\limits_{i=n-1-m}^{n-1}P_{[i+m-(n-1)][i]}
 \Big),\ T_{2m+1}=T_{2m+1}'\oplus P_{[0][0]}\\
&(m=0,\dots,n-2).
\end{align*}

The following description of the terms $Q_t$ ($t\geq 0$) of the bimodule resolution of the algebra $A$ is represented in  \cite{Volk4}.

1. If $n$ is even then $Q_{2m}=T_{2m}$ for $0\leq m\leq n-1$, $Q_{4m+1}=T_{4m+1}$ for $0\leq m\leq \frac{n-2}{2}$ and $Q_{4m+3}=T'_{4m+3}$ for $0\leq m\leq \frac{n-4}{2}$. Moreover, $Q_{l(2n-1)+t}=Q_t$ for $0\leq t\leq 2n-2$.

2. If $n$ is odd then $Q_{2m}=T_{2m}$ for $0\leq m\leq n-1$, $Q_{4m+1}=T_{4m+1}$, $Q_{4m+3}=T'_{4m+3}$ for $0\leq m\leq \frac{n-3}{2}$,
$Q_{2n-1+2m}=T_{2m}$ for $0\leq m\leq n-1$, $Q_{2n-1+4m+1}=T'_{4m+1}$, $Q_{2n-1+4m+3}=T_{4m+3}$ for $0\leq m\leq \frac{n-3}{2}$.
Moreover, $Q_{l(4n-2)+t}=Q_t$ for $0\leq t\leq 4n-3, l \in \mathbb{N}$.

Moreover, the fact that $\Omega^{4n-2}_{A^e}(A)\cong A$ for all $n$ and $\Omega^{2n-1}_{A^e}(A)\cong {}_{\sigma}A\cong A_{\sigma}$ for even $n$ as $ A^e$-modules is noted in the same paper.

It follows from the corollary \ref{cor_scydim} that in order to compute the stable Calabi-Yau dimension we have to find the least $t\geq 0$ such that $\Omega^{t+1}_{ A^e}(A)\cong A_{\varphi}$ where $\varphi$ is a stably inner automorphism (the Nakayama automorphism is equal to identity here). It is clear that $Q_{t+1}=(Q_0)_{\varphi}\cong T_0$  and $Q_{t+2}=(Q_1)_{\varphi}\cong T_1$ in this case.
It follows that that the Calabi-Yau dimension of the algebra $A$ is equal to $4n-3$ if $n$ is odd and equal to either $2n-2$ or $4n-3$ if $n$ is even. The following lemma allows to find exact answer for even $n$.

\begin{lemma}
The automorphism $\sigma$ is not stably inner.
\end{lemma}

\begin{proof} For elements $m_1,\dots,m_t$  of $A$-module $M$, we denote by $\langle m_1,\dots,m_t\rangle$ the submodule of $M$ generated by them. Let us consider the following modules: $M_1=P_1/\langle \nu_1\rangle$, $M_2=P_0/\langle \beta^2\rangle$, $M_3=P_0/\langle \beta^2,\mu_{n-1}\beta\rangle$, $M_4=(P_0\oplus P_{n-1})/\langle \beta+\alpha_{n-1},\beta^2\rangle$, $M_5=P_0/\langle \beta\rangle$ and $M_6=P_0/\langle \beta^3\rangle$. We define homomorphisms $f_i:M_i\rightarrow M_{i+1}$ ($1\leq i\leq 4$) and $f_5:M_4\rightarrow M_6$ by the following formulas:
$
f_1(e_1)=\alpha_0,$ $f_2(e_0)=e_0,$ $f_3(e_0)=e_0,$ $f_4(e_0)=e_0,$ $f_4(e_{n-1})=0,$ $f_5(e_0)=\beta,$ $f_5(e_{n-1})=\mu_{n-1}.
$
Let us assume that $\sigma$ is a stably inner automorphism. Then there is a collection of isomorphisms of $A$-modules $\eta_i:M_i\rightarrow (M_i)_{\sigma}$ such that equalities $\underline{{\rm res}}_{\sigma}\big(\underline{f_i}\big)\underline{\eta_i}=\underline{\eta_{i+1}}\underline{f_i}$ for $1\leq i\leq 4$ and the equality $\underline{{\rm res}}_{\sigma}\big(\underline{f_5}\big)\underline{\eta_4}=\underline{\eta_6}\underline{f_5}$ hold in the category $\underline{{\rm mod}}$-$A$. It is clear that $(M_i)_{\sigma}=M_i$ for $1\leq i\leq 5$ and $\underline{{\rm res}}_{\sigma}\big(\underline{f_i}\big)=\underline{f_i}$ for $1\leq i\leq 4$. Moreover, it can be shown that
$$
\begin{aligned}
&{\rm End}_A(M_1)={}_k\langle {\rm Id}_{M_1}\rangle, {\rm End}_A(M_2)={}_k\langle {\rm Id}_{M_2}, \eta_{2,1}\rangle, {\rm End}_A(M_3)={}_k\langle {\rm Id}_{M_3}, \eta_{3,1}\rangle,\\
&{\rm End}_A(M_4)={}_k\langle {\rm Id}_{M_4}, \eta_{4,1}, \eta_{4,2}\rangle, {\rm End}_A(M_5)={}_k\langle {\rm Id}_{M_5}\rangle,
\end{aligned}
$$
where $\eta_{2,1}(e_0)=\beta$, $\eta_{3,1}(e_0)=\beta$, $\eta_{4,1}(e_0)=\beta$, $\eta_{4,1}(e_{n-1})=0$, $\eta_{4,2}(e_0)=0$, $\eta_{4,2}(e_{n-1})=\mu_{n-1}$.

Then we can assume that $\eta_1={\rm Id}_{M_1}$. Suppose that $\eta_2=c_1{\rm Id}_{M_2}+c_2\eta_{2,1}$, where $c_1, c_2\in k$.
Then the map  $f_1+\eta_2f_1$ which maps $e_1$ to $(c_1+1)\alpha_0+c_2\alpha_0\beta$ must go  through the canonical projection $P_0\rightarrow M_2$. But it is easy to show that ${\rm Hom}_A(M_1,P_0)=0$, i.e. $c_1=1$, $c_2=0$.

Let $\eta_3=c_1{\rm Id}_{M_3}+c_2\eta_{3,1}$ for some $c_1, c_2\in k$. Then the map $f_2+\eta_3f_2$ which maps $e_0$ to $(c_1+1)e_0+c_2\beta$ must go through the canonical projection $\rho_3:P_0\rightarrow M_3$. It is easy to verify that ${\rm Hom}_A(M_2,P_0)={}_k\langle\theta_{2,1},\theta_{2,2}\rangle$ where $\theta_{2,1}$ and $\theta_{2,2}$ are the maps which map $e_0$ to $\beta^2$ and $\beta^3$ respectively, i.e. every map from $M_2$ to $M_3$ which goes through $\rho_3$ equals 0. Consequently, $\eta_3={\rm Id}_{M_3}$.

Let $\eta_4=c_1{\rm Id}_{M_4}+c_2\eta_{4,1}+c_3\eta_{4,2}$, $\eta_5=c{\rm Id}_{M_5}$ where $c_1, c_2, c_3, c\in k$. Then the map $f_3+\eta_4f_3$ which maps $e_0$ to $(c_1+1)e_0+c_2\beta$ must go through the  canonical projection $\rho_4:P_0\oplus P_{n-1}\rightarrow M_4$. It is easy to show that ${\rm Hom}_A(M_3,P_0\oplus P_{n-1})={}_k\langle\theta_{3,1},\theta_{3,2},\theta_{3,3}\rangle$ where $\theta_{3,1}$, $\theta_{3,2}$ and $\theta_{3,3}$ are the maps which map $e_0$ to $\beta^2$, $\beta^3$ and $\beta\alpha_{n-1}$ respectively, i.e. every map from $M_3$ to $M_4$ which goes through $\rho_4$ is equal to $0$. Hence, $\eta_4={\rm Id}_{M_4}+c_3\eta_{4,2}$, i.e. the map $f_4\eta_4+\eta_5f_4$ maps $e_0$ and $e_{n-1}$ to $(1+c)e_0$ and $c_3\mu_{n-1}$ respectively. This map must go through the canonical projection $\rho_5:P_0\rightarrow M_5$. It is easy to verify that ${\rm Hom}_A(M_4,P_0)={}_k\langle\theta_{4,1},\theta_{4,2}\rangle$ where $\theta_{4,1}(e_0)=\beta^2$, $\theta_{4,1}(e_{n-1})=\mu_{n-1}\beta$, $\theta_{4,2}(e_0)=\beta^3$, $\theta_{4,2}(e_{n-1})=0$, i.e. every map from $M_4$ to $M_5$ which can go through $\rho_5$ is equal to $0$. It follows that $\eta_4={\rm Id}_{M_4}$.

Hence, we have $\underline{f_5}=\underline{\eta_6^{-1}}\underline{{\rm res}}_{\sigma}\big(\underline{f_5}\big)$. The map $\eta_6^{-1}$ maps $e_0$ to $c_1e_0+c_2\beta+c_3\beta^2$ for some $c_1,c_2,c_3\in k$. Then $f_5+\eta_6^{-1}{\rm res}_{\sigma}(f_5)$ maps $e_0$  and $e_{n-1}$ to
$$
\begin{aligned}
\beta+\eta_6^{-1}(\beta)&=\beta+\eta_6^{-1}\big(\sigma(\beta+\beta^2)e_0\big)=\beta+(\beta+\beta^2)(c_1e_0+c_2\beta+c_3\beta^2)\\
&=(c_1+1)\beta+(c_1+c_2)\beta^2,
\end{aligned}
$$
and $(c_1+1)\mu_{n-1}+c_2\mu_{n-1}\beta$ respectively. This map must be represented as $\rho_6\theta$ where $\theta\in{\rm Hom}_A(M_4,P_0)$ and $\rho_6:P_0\rightarrow M_6$ is the canonical projection. It is mentioned above that ${\rm Hom}_A(M_4,P_0)={}_k\langle\theta_{4,1},\theta_{4,2}\rangle$, i.e. $\rho_6\theta$ maps $e_0$ and $e_{n-1}$ to $d_1\beta^2+d_2\beta^3$ and $d_1\mu_{n-1}\beta$ respectively for some $d_1, d_2\in k$. Thus, $c_1=1$, $c_1+c_2=d_1$, $c_2=d_1$. The contradiction proves the lemma.
\end{proof}

The following proposition follows from the above argument:

\begin{proposition}
Let $A$ be a nonstandard self-injective algebra of finite representation type of type $(D_{3n}, \frac{1}{3}, 1)$. Then the stable Calabi-Yau dimension of the algebra $A$ is equal to $(4n-3)$.
 \end{proposition}

\end{document}